\documentclass[11pt]{article}
\usepackage{}
\usepackage[total={6in, 9in}]{geometry}
\usepackage{amssymb}
\usepackage{mathrsfs}
\usepackage[tbtags]{amsmath}
\usepackage{amsfonts,amsthm,amssymb}
\usepackage{algorithm,algorithmic}
\usepackage{amsfonts}
\usepackage{caption}
\usepackage{float}
\usepackage{graphicx}
\usepackage{graphics}
\usepackage{subfigure}
\usepackage{enumerate}
\usepackage[numbers,sort&compress]{natbib}
\usepackage[
pdfauthor={},
pdftitle={},
pdfstartview=XYZ,
bookmarks=true,
colorlinks=true,
linkcolor=blue,
urlcolor=blue,
citecolor=blue,
bookmarks=true,
linktocpage=true,
hyperindex=true
]{hyperref}

\newtheorem{lem}{Lemma}
\newtheorem{thm}[lem]{Theorem}
\newtheorem{cor}[lem]{Corollary}

\newtheorem{defi}[lem]{Definition}
\newtheorem{con}[lem]{Conjecture}

\begin{document}
	
	\title{Hamiltonian cycles in $ 15 $-tough ($ P_{3}\cup 3P_{1} $)-free graphs}

	\author{Hui Ma, Lili Hao, Weihua Yang\footnote{Corresponding author. E-mail: ywh222@163.com; yangweihua@tyut.edu.cn}\\
		\\ \small Department of Mathematics, Taiyuan University of
		Technology,\\
		\small  Taiyuan Shanxi-030024,
		China\\
	}
	
	\date{}
	\maketitle
	
	\emph{\textbf{Abstract.}} 
	A graph $ G $ is called $ t $-tough if $  \left|S\right|\geq t\cdot w\left(G-S\right)$ for every cutset $ S $ of $G$.  Chv\'atal conjectured that there exists a constant $ t_{0} $ such that every $ t_{0} $-tough graph has a hamiltonian cycle. Gao and Shan have proved that every $7$-tough $(P_{3}\cup 2P_{1})$-free grah is hamiltonian. 
	In this paper, we confirm this conjecture for $ (P_{3}\cup 3P_{1}) $-free graphs.

	
	\section{Introduction}

	All graphs considered  in the paper are finite, undirected and simple. Let $ V(G) $ and $ E(G) $  denote the vertex set and the edge set of the graph $G$, respectively. For a vertex $ v\in V(G) $, let $  N_{G}(v) $ denote the set of neighbors of $ x $ in $ G $. Let  $ \delta(G)=\min\left\{\left|N_{G}(v): v\in V(G)\right|\right\} $. We denote by $ \alpha(G) $ the independence number of $ G $ and by $ \kappa(G) $ the  connectivity of $ G $. For any set $ X\subseteq V(G) $, we denote by $ G[X] $ the induced subgraph on $ X $, by $ G-X $ the subgraph $ G[V(G)\backslash X] $, and $ w(G-X) $ denotes the number of connected components of $ G-X $. For any two sets $ X, Y\subseteq V(G)$ or any two subgraphs $ X, Y$  of $ G $,  $ N_{X}(Y) $ denotes the set of all vertices in $X-Y$ adjacent to some vertices in $ Y $, and $E[X, Y]$ denotes the set of edges with one end in $X$ and the other in $Y$.  Let $ Q $ be a cycle or a path of $G$. Assume $ Q $ has a given direction. For $ x\in V(Q) $, $x^{+}$ and $x^{-}$ denote its successor and prodecessor, respectively. For $x, y\in V(Q)$,  $x\mathop{Q}\limits^{\rightarrow}y$ denotes the path  from $x$ to $y$ following  the given direction of $Q$. Likewise, $x\mathop{Q}\limits^{\leftarrow}y$ denotes the inverse path from $x$ to $y$. For a subgraph $H$ of $G$ and a subset $U$ of $V(G)$, if each vertex in $U$ is adjacent to all the vertices in $H$, we write $U\sim H$;  otherwise, we write $U\nsim H$. If $U=\left\{u\right\}$, we write $u\sim H$ and $u\nsim H$ for $U\sim H$ and $U\nsim H$, respectively. For two integers $a$ and $b$ with $a\leq b$, define $[a, b]=\left\{i\in \mathbb{Z} : a\leq i\leq b\right\}$.



\begin{defi}
	Let $ t $ be a nonnegative real number.	A graph $ G $ is called $ t $-tough if $ \lvert S\rvert\geq t\cdot w(G-S) $ holds for every cutset $ S $ of $ G $. The toughness of a graph $ G $, denoted by $ \tau(G) $, is the largest $ t $ such that $ G $ is $ t $-tough, if $ G $ is noncomplete, and is $ +\infty $ if $G$ is complete.		\end{defi}

It is easy to see that every $t$-tough noncomplete graph $ G $ is $\lceil 2t\rceil $-connected and every hamiltonian graph is 1-tough.  In 1973,  Chv\'atal posed a conjecture.

\begin{con}[\cite{Chvatal}]\label{C1}
	There exists a constant $ t_{0} $ such that every $ t_{0} $-tough graph on at least three vertices is hamiltonian.
\end{con}

A graph is said to be  $ R $-\textit{free} if it contains no copy of  $ R $ as an  induced subgraph. Take $\left(R_{1}\cup R_{2}\right)$ as the vertex-disjoint union of two graphs $R_{1}$ and $R_{2}$, and $kP_{1}$ as $k$ isolatd vertices. Recent research has confirmed the conjecture  for many classes of $R$-free graphs including $2K_{2}$-free graphs \cite{Broersma 2, Shan2020, Ota}, $(P_{2}\cup P_{3})$-free graphs \cite{Shan}, $(P_{2}\cup kP_{1})$-free graphs \cite{Shi, Xu},  $(P_{4}\cup P_{1})$-free graphs \cite{Shan3} and $ (P_{3}\cup \ell P_{1}) $-free graphs for $\ell\leq 2$ \cite{Li, Gao}. Note that when $\left|R\right|\geq 6$, there are very few results proving the conjecture \ref{C1} for $R$-free graphs. In this paper, we prove that the conjecture \ref{C1} is true for $(P_{3}\cup 3P_{1})$-free graphs.

\begin{thm}\label{result 12}
	Every $ 15 $-tough $ (P_{3}\cup 3P_{1}) $-free graph on at least 3 vertices is hamiltonian.
\end{thm}



\section{Proof of Theorem \ref{result 12}}
By the structure of $(P_{3}\cup kP_{1})$-free graphs, we can obtain the following lemma.
\begin{lem}\label{result 0}
	Let $k\geq 1$ be an integer, $ G $ be a $ (P_{3}\cup kP_{1}) $-free graph, and $S\subseteq V(G)$. Let \( S_1  \) denote the set of vertices in \( S \) that are adjacent to vertices in exactly one component of \( G - S \), and \( S_2  \) denote the set of vertices in \( S \) that are adjacent to vertices in at least two components of \( G - S \). 
	
	{\rm(i)} Assume $ w\left(G-S\right)=k $. Then at most one component of $ G-S $ is noncomplete. Moreover, if $k\geq 2$ and $G-S$ has a noncomplete component, then the noncomplete component of $G-S$ is $(P_{3}\cup P_{1})$-free.
	
	{\rm(ii)} If $ w(G-S)\geq k+1 $, then each component of $ G-S $ is complete and
 each vertex in $ S_{1}\cup S_{2} $ is adjacent to all the vertices in some component of $ G-S $.
	
	{\rm(iii)} Assume $ w\left(G-S\right)\geq k+2 $. Then $ w(G-S_{2})\geq w(G-S) $ and each vertex in $ S_{2} $ is adjacent to all the vertices in at least $ w\left(G-S_{2}\right)-k+1 $ components of $ G-S_{1} $. Consequently, any two vertices in $S_{2}$ are adjacent to all the vertices in $ w\left(G-S\right)-2(k-1) $ common components of $G-S$.
\end{lem}

\begin{lem}[\cite{V}]\label{result 1}
	Let $ G $ be a bipartite graph with partite sets $ X $ and 	$ Y$, and $ f $ be a function from $ X $ to the set of positive integers. If for every $ S\subseteq X $, $ \left|N_{G}(S)\right|\geq \sum_{v\in S}f(v) $, then $ G $ has a subgraph $ H $ such that $ X\subseteq V(H) $, $ d_{H}(v)=f(v) $ for every $ v\in X $, and $ d_{H}(u)=1 $ for every $ u\in Y\cap V(H) $. 
\end{lem}

A $ K_{1, s} $-\textit{matching} in a graph is a set of vertex-disjoint copies of $ K_{1, s} $. The vertices of degree at least 2 in such a matching are called \textit{centers.}
Let \( G \) be a graph, \( X \subseteq V(G) \) be nonempty, and \( Y \subseteq N_G(X) \). For some integer \( s \geq 1 \), \( G[X, Y] \) is defined as a \textit{generalized} \( K_{1, 2s} \) with center \( X \) if \( |Y| = 2s \), and if \( |X| \geq 2 \), there exist partitions \( \{X_1, X_2\} \) of \( X \) and \( \{Y_1, Y_2\} \) of \( Y \) such that \( Y_i \subseteq N_G(X_i) \cap Y \) and \( |Y_i| = s \)  for each \( i \in \{1, 2\} \). Here, vertices in \( Y \) are called \textit{partners} of vertices in \( X \). Let \( S \) be a cutset of \( G \), and let \( D_1, D_2, \ldots, D_{\ell} \) be the components of \( G - S \) for some integer \( \ell \geq 2 \). For any \( i \in [1, \ell] \), if there exists a subset \( S_i \subseteq N_G(D_i) \cap S \) such that \( G[V(D_i), S_i] \) is a generalized \( K_{1, 2s} \) with \( V(D_i) \) as its center, and if \( S_i \cap S_j = \emptyset \) for distinct \( i, j \in [1, \ell] \), then \( G \) is said to have a \textit{generalized \( K_{1, 2s} \)-matching} centered at the components of \( G - S \). In the following, we call $S$ a \textit{proper cutset} of $G$ if each vertex in $S$ has neighbors in at least two components of $G-S$.	

\begin{lem}\label{result 6}
Let $ G $ be a  $ t $-tough $ (P_{3}\cup 3P_{1}) $-free graph for some $ t\geq 2 $, $S$ be a  cutset of $G$, and let $s=\lfloor \frac{t}{2}\rfloor$.
If $w\left(G-S\right)\geq 5$, then $ G $ has a generalized $ K_{1, 2s} $-matching centered at the components of \( G - S \).
\end{lem}
\begin{proof}
 Let $\mathcal{l}=w\left(G-S\right)$, and let $S_{1}$ be the set of vertices that are adjacent to vertices in at least two components of $G-S$. Then by Lemma \ref{result 0}(iii), we have $w\left(G-S_{1}\right)\geq w\left(G-S\right)=\mathcal{l}$.  Since $ G $ is $ t $-tough, we have $ \left|S\right|\geq \left|S_{1}\right|\geq  t\mathcal{l}\geq 2s\mathcal{l} $ and $ \lvert N_{S}(D_{i})\rvert\geq\kappa(G)\geq 2t\geq 4s $ for each $ i\in [1, l] $. 

Let $ D $ be a nontrivial component of $ G-S $ (if exists), and let $ W=N_{S}(D) $. We claim that there exist two partitions $ \left\{D^{1}, D^{2} \right\} $ of $ V(D) $ and $ \left\{W^{1},  W^{2}\right\} $ of $ W $ such that $ \min\left\{\left|W^{1}\right|, \left|W^{2}\right|\right\}\geq 2s $, where $ W^{1}\subseteq N_{S}(D^{1}) $ and $ W^{2}\subseteq N_{S}(D^{2}) $. If $ \left|V(D)\right|\geq 4s $, then by $ \kappa(G)\geq 4s $, there exist at least $4s$ pairwise disjoint edges between $V(D)$ and $W$, thus the claim holds. Now we assume $ 2\leq \left|D\right|\leq 4s-1 $. Let $ \left\{D^{1j}, D^{2j} \right\} $ and $ \left\{W^{1j}, W^{2j}\right\} $ be two partitions of $ V(D) $ and $ W $, respectively,  such that $ \left|\left|W^{1j}\right|-\left|W^{2j}\right|\right| $ is minimum, where $ W^{1j}\subseteq N_{S}(D^{1j}) $ and $ W^{2j}\subseteq N_{S}(D^{2j}) $.  Without loss of generality, we assume $\left|W^{1j}\right|\geq \left|W^{2j}\right|$. Suppose that $\left|W^{2j}\right|\leq 2s-1 $. By $ \left|W\right|\geq \kappa(G)\geq 4s $, we have $\left|W^{1j}\right|\geq 2s+1 $. By the minimum of $ \left|\left|W^{1j}\right|-\left|W^{2j}\right|\right| $, we have 
$ N_{S}(D^{2j})\cap W^{1j}=\emptyset $. So $ \left|D^{1j}\cup W^{2j}\right|\geq \left|N_{G}(D^{2j})\right|\geq \kappa(G)\geq 4s $ and then $ \left|D^{1j}\right|\geq 2s+1 $. As $\left|V(D)\right|\leq 4s-1$, we have $ \left|D^{2j}\right|\leq 2s-2 $.  Let $ u $ be a vertex in $ D^{1j} $ such that $ \left|N_{W^{1j}}(u)\right|\geq 1 $, $ D^{11}=D^{1j}\backslash\left\{u\right\} $, $ D^{21}=D^{2j}\cup\left\{u\right\} $, $W^{11}=W^{1j}\backslash N_{W^{1j}}(u) $ and $ W^{21}=W^{2j}\cup N_{W^{1j}}(u) $. Then $\left|D^{21}\right|\leq 2s-1$. By the minimum of $ \left|\left|W^{1j}\right|-\left|W^{2j}\right|\right| $, this implies that $\left|W^{11}\right|\leq 2s-1 $. Let $\left|W^{11}\right|= 2s-k  $. Then $\left|W^{21}\right|\geq 2s+k$. By $ \left|N_{G}(D^{11})\right|\geq \kappa(G)\geq 4s$, we can obtain$$ \left|N_{S}(D^{11})\cap W^{21}\right|\geq 4s-\left(\left|W^{11}\right|+\left|D^{21}\right|\right)\geq 4s-(2s-k)-(2s-1)=k+1 .$$ 
Let $ X $ be a set of $ k $ vertices in $N_{S}(D^{11})\cap W^{21}$,  $ W^{13}=W^{11}\cup X $ and $ W^{23}=W^{21}\backslash X $. Clearly, $ \left\{W^{13}, W^{23}\right\} $ is a partition of $W$ with $ W^{13}\subseteq N_{S}(D^{11}) $ and $ W^{23}\subseteq N_{S}(D^{21}) $. Note that $ \left|W^{13}\right|=2s $ and $ \left|W^{23}\right|\geq 2s $, this implies that $\left|\left|W^{13}\right|-\left|W^{13}\right|\right|<\left|\left|W^{1j}\right|-\left|W^{2j}\right|\right|$, a contradiction.	  

Let $ D_{1}, D_{2}, \cdots, D_{\mathcal{l}} $ be  the components of $ G-S $. For each $ i\in [1, \mathcal{l}]$, if $ D_{i} $ is nontrivial, let $ \left\{D_{i}^{1}, D_{i}^{2}\right\} $ and $ \left\{W_{i}^{1}, W_{i}^{2}\right\} $ be  partitions of $ D_{i} $ and $ N_{S}(D_{i}) $, respectively, such that $ \min\left\{\left|W_{i}^{1}\right|, \left|W_{i}^{2}\right|\right\}\geq 2s $, where $W_{i}^{1}\subseteq N_{S}(D_{i}^{1})$ and $W_{i}^{2}\subseteq N_{S}(D_{i}^{2})$.
We will construct a bipartite graph in the following. 
For each nontrivial component $ D_{i} $, we contract all vertices in $ D_{i}^{1} $ and $ D_{i}^{2} $ into two vertices $ a_{i} $ and   $ b_{i} $, respectively. For each $ D_{i} $ containing exactly one vertex $ w_{i} $, we split $ w_{i} $ into vertices $ a_{i} $ and  $ b_{i} $, and distribute the neighbors of $ w_{i} $ in $ G $ between $ a_{i} $ and $ b_{i} $ such that $ \left|N_{S}(a_{i})\right|\geq 2s $, $ \left|N_{S}(b_{i})\right|\geq 2s $ and $ N_{S}(a_{i})\cap N_{S}(b_{i})=\emptyset $.	  
Let $ T=\left\{a_{i}, b_{i}: 1\leq i\leq \mathcal{l}\right\} $ and let $ G_{1}=G[S, T] $ be a bipartite graph with $ V(G_{1})=S\cup T $ and  $ E(G_{1})=E(S,  \bigcup_{i=1}^{\ell} V(D_{i}))  $. Note that each vertex in $T$ has at least $2s$ neighbors in $ G_{1} $. To complete the proof, it is sufficient to prove that $ G_{1} $ has a $ K_{1, s} $-matching saturating $ T $. Let $ T^{\prime}\subseteq T $. If $ \lvert T^{\prime}\rvert\leq 2 $, then $ \lvert  N_{G_{1}}(T^{\prime})\rvert\geq 2s\geq s\lvert T^{\prime}\rvert $. Now assume $ \lvert T^{\prime}\rvert\geq 3 $. 
If there exist three vertices in $ T^{\prime} $ obtained from three components of $ G-S $, then Lemma \ref{result 0}(iii) implies that $ S_{1}\subseteq N_{G_{1}}(T^{\prime}) $, so $ s\lvert T^{\prime}\rvert\leq 2s\ell\leq t\ell\leq \left|S_{1}\right|\leq \left| N_{G_{1}}(T^{\prime})\right|$. If the vertices in $ T^{\prime} $ obtained from at most 2 components of $ G-S $, then $ \lvert T^{\prime}\rvert\leq 4  $. Since each component of $ G-S $ has at least $ 4s $ neighbors in $S$,   it is easy to see that $ \lvert  N_{G_{1}}(T^{\prime})\rvert\geq 4s\geq s\lvert T^{\prime}\rvert $. Thus by Lemma \ref{result 1}, $ G_{1} $ has a $ K_{1, s} $-matching saturating $T$.
\end{proof}

As $\kappa(G)\geq \lceil 2t\rceil$, we can obtain the following corollary.
\begin{cor}\label{C2}
Let $G$ be $t$-tough $(P_{3}\cup 3P_{1})$-free graph with $t\geq 4$, and $S$ be a cutset of $G$. Then $G$ has a generalized $K_{1, 2}$-matching centered at the components of $G-S$.
\end{cor}

\begin{thm}[\cite{Li}]\label{result 10}
Let $ R $ be an induced subgraph of $P_{4}$, $ P_{3}\cup P_{1} $ or $P_{2}\cup 2P_{1}$. Then every $1$-tough $R$-free graph on at least three vertices is hamiltonian.
\end{thm}	
\begin{lem}[\cite{Gao}]\label{result 5}
Let $ G $ be a more than $1$-tough $(P_{3}\cup P_{1}) $-free graph on $ n\geq 3 $ vertices. Then $ G $ is hamiltonian connected.
\end{lem}

\begin{lem}\label{result 9}
Let $ G $ be a  $ (P_{3}\cup 2P_{1}) $-free graph, and $k(G)$ be the minimum number of pairwise disjoint paths covering all the vertices in $G$.
If $\tau(G)\geq 1$, then $k(G)\leq 2$. If $\tau(G)<1$, there exists a cutset $W\subseteq V(G)$ such that $k(G)\leq w\left(G-W\right)-\left|W\right|\leq \alpha(G)$. 
\end{lem}  

\begin{proof}
We assume there exists a noncomplete component in $G$. By Lemma \ref{result 0} (i) and (ii), we have that $1\leq w\left(G\right)\leq 2$ and there exists exactly one noncomplete component $H$ in $G$. Let $k(H)$ be the minimum number of pairwise disjoint paths in $H$ covering all the vertices in $H$.
If $ w\left(G\right)=1 $, then $H=G$. If $w\left(G\right)=2$, then $k(G)=k(H)+1$, $ \alpha(G)=\alpha(H)+1 $ and $w\left(G-Y\right)=w\left(H-Y\right)+1$ for any subset $Y\subseteq Y\subseteq V(H)$. Thus it is sufficient to prove that this lemma holds for $H$. 

\noindent{\bf Case 1.} $\tau(H)\geq 1$.

Let $ P $ be a longest path of $ H $, and let two ends of $P$ be $x_{0}$ and $x_{l+1}$. Assume the direction of $P$ is from $x_{1}$ to $x_{l-1}$. We will prove that if $ V(H)\backslash V(P)\neq\emptyset $, then  $G\left[ V(H)\backslash V(P)\right]$ is a clique.
Let $X=\left\{x_{1}, \dots,  x_{l}\right\}$ be a set of neighbors in $P$ of vertices in $H-V(P)$.
Clearly,  $l\geq \kappa(H)\geq 2\tau(H)\geq 2 $.  Since $ P $ is a longest path of $ H $, we have that $ x_{0}x_{l+1}\notin E(H) $ and $X\cap \left\{x_{1}, x_{l+1}\right\}=\emptyset $.
Since $ H $ is $ \left(P_{3}\cup 2P_{1}\right) $-free, each component of $ H-V(P) $ is complete.  Suppose to the contrary that $ w\left(H-V(P)\right)\geq 2 $.

Let $ X^{+}=\left\{x_{1}^{+}, \cdots,  x_{l}^{+}\right\} $ and $ X^{-}=\left\{x_{1}^{-}, \cdots,  x_{l}^{-}\right\} $.
Note that $ N_{H}(x_{0})\cap X^{+}=\emptyset $. Otherwise, there exists $ x_{j}\in X $ such that $ x_{j}^{+}x_{0}\in E(G) $. Let $u_{j}\in N_{H-V(P)}(x_{j})$. Then $ u_{j}x_{j}Px_{0}x_{j}^{+}Px_{l+1} $ is a longer path than $P$ in $H$, a contradiction. Similarly, we can obtain $ N_{H}(x_{l+1})\cap X^{-}=\emptyset $.

Let $ V_{0}=V(x_{0}Px_{1}^{-}) $,  $ V_{l}=V(x_{l}^{+}Px_{l+1}) $ and  $ V_{i}=V(x_{i}P_{1}x_{i+1})\backslash\left\{x_{i}, x_{i+1}\right\} $  for each $ i\in [1, l-1] $.	We claim that $ V_{i}\neq \emptyset $ for each $ i\in [0, l] $.  For otherwise, 
there exists $ {j}\in [1, l]$ such that $x_{j}^{+}=x_{j+1} $. Let $ u_{j}\in N_{H-V(P)}(x_{j}) $ and $u_{j+1}\in N_{H-V(P)}(x_{j+1})$.
Since $ P $ is a lonest path of $H$, 
this implies that $ u_{j} $ and $u_{j+1}$ are in two distinct components of $H-V(P)$.
Consider two paths $ u_{j}x_{j}x_{j+1} $ and $u_{j+1}x_{j+1}x_{j}$, and two nonadjacent vertices $ x_{0}, x_{l+1} $. Since $ H $ is $ (P_{3}\cup 2P_{1}) $-free, $ N_{G}(x_{0})\cap X^{+}=\emptyset $ and $ N_{H}(x_{l+1})\cap X^{-}=\emptyset $, we have $ x_{0}x_{j}\in E(H) $ or $ x_{l+1}x_{j+1}\in E(H) $.
Then the path $u_{j}x_{j}x_{0} $ and two vertices $ u_{j+1}, x_{l+1} $, or the path $u_{j+1}x_{j+1}x_{l+1} $ and two vertices $ u_{j}, x_{0} $ induce a copy of $ P_{3}\cup 2P_{1} $ in $ H $, a contradiction.

We will prove that for distinct $ i, j\in [0, l] $, we have $E[V_{i}, V_{j}]=\emptyset $.	Suppose to the contrary that there exist $i, j\in [0, l]$ such that $E\left(V_{i}, V_{j}\right)\neq \emptyset  $. By the $ \left(P_{3}\cup 2P_{1}\right) $-freeness of $ G $ and $ w\left(H-V(P)\right)\geq 2 $, we know that $ H[V_{i}\cup V_{j}] $ is complete. Since $ x_{0}x_{l+1}\notin E(H) $,  this implies that $i\neq 0$ or $ j\neq l $. Moreover, since $ N_{H}(x_{0})\cap X^{+}=\emptyset $ and $ N_{H}(x_{l+1})\cap X^{-}=\emptyset $, we have $ i, j\in [1, l-1] $. Assume $ i<j $.
If there exists a component $D$ of $H-V(P)$ such that $ N_{D}(x_{i})\neq \emptyset $ and $ N_{D}(x_{j})\neq \emptyset $, then there exists a path $Q$ in $H-\left(V(P)\backslash\left\{x_{i}, x_{j}\right\}\right)$ from $x_{i}$ to $x_{j}$. Since $ H[V_{i}\cup V_{j}] $ is complete, we have  $x_{i}^{+}x_{j}^{+}\in E(H)$. Then $ x_{0}Px_{i}Qx_{j}Px_{i}^{+}x_{j}^{+}Px_{l+1}$ is a longer than $ P $, a contradiction. Otherwise, let $ u_{i}\in N_{H-V(P)}(x_{i}) $ and $u_{j}\in N_{H-V(P)}(x_{i})$, then we have $ x_{j}u_{i}, x_{i}u_{j}\notin E(G) $. Consider the path $u_{i}x_{i}x_{i}^{+}$ and two vertices $x_{0}, u_{j}$. Since $ H $ is $ \left(P_{3}\cup 2P_{1}\right) $-free and $ N_{H}(x_{0})\cap X^{+}=\emptyset $, we have  $ x_{i}x_{0}\in E(H) $. Consider the path $u_{i}x_{i}x_{0}$ and two vertices $u_{j}, x_{l+1}$. Since $ H $ is $ \left(P_{3}\cup 2P_{1}\right) $-free and $ N_{H}(x_{0})\cap X^{+}=\emptyset $, we have  $ x_{i}x_{l+1}\in E(H) $. Now $ u_{j}x_{j}Px_{i}^{+}x_{j}^{+}Px_{l+1}x_{i}Px_{0} $ is a longer path in $H$  than $P $, a contradiction.

Since  $V_{i}\neq\emptyset$ and $E[V_{j}, V_{j}]=\emptyset$ for any two distinct numbers $i, j\in [0, l]$, this implies that $w\left(H-X\right)\geq l+3 $, so $ \frac{\left|X\right|}{w\left(H-X\right)}\leq \frac{l}{l+3}<1 $, a contradiction.

\noindent{\bf Case 2. $0<\tau(H)<1$.}

Let $S$ be a tough set of $H$. Then $ \left|S\right|\leq w\left(H-S\right). $ $S$ is a proper cutset of $H$.

\noindent{\bf Case 2.1.} Each component of $H-S$ is complete. 

	Let $ \ell=w\left(H-S\right) $. Then $\left|S\right|\leq \mathcal{l}$. If $\ell=2$, then clearly, we have $ k(H)=1\leq \ell-\left|S\right|<\alpha(H) $. If $ \ell=3$, then as $S$ is a proper cutset of $H$ and each vertex in $S$ is adjacent to all the vertices in some component of $H-S$ by Lemma \ref{result 0}(ii), we can obtain $k(H)\leq \ell-\left|S\right|<\alpha(H)$. 	If $ \ell\geq 4 $, then by Lemma \ref{result 0}(iii), each vertex in $S$ is adjacent to all the vertices in at least $ \ell-1$ components of $H-S$, so we have $k(H)\leq\ell-\left|S\right|<\alpha(H)$. 

\noindent{\bf Case 2.2} There exists a noncomplete component in $H-S$. 

By Lemma \ref{result 0}(i) and (ii), we have that $w\left(H-S\right)=2$ and there exists exactly one noncomplete component $D$ of $H-S$ and $D$ is $(P_{3}\cup P_{1})$-free.  $\tau(H)<1$ implies that $\left|S\right|=1$.  If $\tau(D)\geq 1$, then by Theorem \ref{result 10}, we have $k(H)\leq 1\leq\alpha(H)$. If $\tau(D)<1$, then let $T$ be a tough set of $D$. Let $S_{1}=\left\{x\in S: N_{D}(x)\cap (V(D)-T)\neq\emptyset\right\}$ and $T_{1}=S_{1}\cup T$. Note that $w\left(H-T_{1}\right)\geq \max\left\{3, \left|T_{1}\right| \right\}$ and each vertex in $T_{1}$  has neighbors in at least two components of $H-T_{1}$. By Lemma \ref{result 0}(ii), each component of $H-T_{1}$ is complete. By a similar argument as Case 2.1, we  obtain $ k(H)\leq w\left(H-T_{1}\right)-\left|T_{1}\right|<\alpha(H) $. 
\end{proof}


\begin{lem}[\cite{Shan1}]\label{result 4}
Let $ t>0 $ be a real number,  $ G $ be
a  $ t $-tough graph on $ n\geq 3 $ vertices and $ C $ be a nonhamiltonian cycle of $ G $. If $ x\in V(G)\backslash V(C) $ satisfies that $ d_{C}(x)>\frac{n}{t+1}-1 $, then $ G $ has a cycle $ C^{\prime} $ such that $ V(C^{\prime})=V(C)\cup\left\{x\right\} $. 
\end{lem}	

\begin{lem}\label{result 11}
Let $ t\geq 8 $ be a real number, $ G $ be a $ t $-tough $ (P_{3}\cup 3P_{1}) $-free graph, and $ S $ be a proper cutset of $ G $. If there are at least 5 components in $G-S$, three of which are nontrivial, and
each vertex in $ S $ has more than $ \frac{n}{t+1}-1 $ neighbors in $ G-S $,  then 
$G$ is hamiltonian.
\end{lem} 
\begin{proof}
We will construct a cycle $ C $ containing all the vertices in  $ G-S$. By Lemma \ref{result 4}, we can insert all the vertices of $S-V(C)$ into $C$. The resulting cycle is a hamiltonian cycle of $G$.

Let $ D_{1}, D_{2}, \dots, D_{l} $ be the components of $ G-S $. By $ \tau(G)\geq 8 $ and Lemma \ref{result 6}, there exist distinct vertices $ x_{1}, y_{1}, x_{2}, y_{2}, \dots,$ $ x_{l}$, $ y_{l} $ in $ S $ such that $ x_{i}, y_{i} $ are partners of $ D_{i} $ for each $i\in [1, l]$. As each $ D_{i} $ is complete, there exists a hamiltonian path $ P_{i} $ in $ D_{i} $ with
the two ends $ a_{i}, b_{i} $ of $ P_{i} $  adjacent to $ x_{i}$ and $y_{i} $, respectively. Note that $ a_{i}=b_{i} $ if $ \left|V(D_{i})\right|=1 $. By Lemma \ref{result 0}(iii), assume $ y_{i}a_{i+1}\in E(G) $ for each $i\in [1, l-3]$, and we can obtain a cycle 
$C_{1}=a_{r}P_{r}b_{r}y_{r}a_{r+1}P_{r+1}b_{r+1}y_{r+1}\dots y_{k-1}a_{k}P_{k}b_{k}y_{k}a_{r}$ (if necessary, swap the labels of components $D_{l-1}$ and $D_{l}$), where $ r\in [1, 3] $ and $k\geq l-3+r\geq l-2$.
Assume $ r=1 $ and $k\leq l-1$.

\noindent{\bf Case 1.} $k=l-1$.

If $ l\geq 6 $, then by Lemma \ref{result 0}(iii), there exists $ j\in [1, l-2] $ such that $ x_{l}\sim D_{j} $ (resp. $D_{j+1}$) and $ y_{l}\sim D_{j+1} $ (resp. $D_{j}$), or $ x_{l}\sim D_{1} $ (resp. $D_{l-1}$) and  $ y_{l}\sim D_{l-1} $ (resp. $D_{1}$). Now replace the path $ b_{j}y_{j}a_{j+1} $ or $a_{1}y_{l-1}b_{l-1}$ on $C_{1}$ with $ b_{j}x_{l}a_{l}P_{l}b_{l}y_{l}a_{j+1} $ or $a_{1}x_{l}a_{l}P_{l}b_{l}y_{l}b_{l-1} $, respectively. Note that the resulting cycle is a desired cycle. 

Now we assume $l=5$. Since $G-S$ has at least 3 nontrivial components, we assume $D_{1}$ is nontrivial.  If $\left|V(D_{5})\right|\geq 3$, let $G_{1}=G-\bigcup_{i=2}^{4}\left\{a_{i}, b_{i}, y_{i}\right\}\cup \left\{y_{1}, a_{1}\right\}$; otherwise, let $G_{1}=G-\bigcup_{i=2}^{4}\left\{a_{i}, b_{i}, y_{i}\right\}\cup \left\{y_{1}\right\}$. Then as $\kappa(G)\geq \lceil 2t\rceil\geq 16$, we have $\kappa(G_{1})\geq 5$. If $\left|V(D_{1})\right|=1$, then there exist 2 internally disjoint paths in $G_{1}$ from $V(D_{5})$ to $V(D_{1})$, whose terminal vertices are distinct and internal vertices belong to $V(G)\backslash\left(V(D_{5})\cup V(D_{1})\right)$. If $\left|V(D_{5})\right|\geq 2$, then there exist 2 disjoint paths in $G_{1}$ from $V(D_{5})$ to $V(D_{1})$, whose internal vertices belong to $V(G)\backslash\left(V(D_{5})\cup V(D_{1})\right)$. Denote these two paths by $Q_{1}$ and $Q_{2}$. Let $V(Q_{1})\cap V(D_{1})=\left\{w\right\}$, $V(Q_{1})\cap V(D_{5})=\left\{u\right\}$, $V(Q_{2})\cap V(D_{1})=\left\{w^{\prime}\right\}$ and  $V(Q_{2})\cap V(D_{5})=\left\{v\right\}$. Let $P_{1}^{\prime}$ be a hamiltonian path of $D_{1}$ from $a_{1}$ to $b_{1}$ such that $ww^{\prime}\in E(P^{\prime})$, $P_{i}^{\prime}$ be a hamiltonian path of $D_{i}-V(Q_{1})\cup V(Q_{2})$ from $a_{i}$ to $b_{i}$ for each $i\in [2, 4]$, and $P_{5}^{\prime}$ be a hamiltonian path of $D_{5}$ from $u$ to $v$. Now replace the path $P_{i}$ on $C_{1}$ with $P_{i}^{\prime}$ for each $i\in [1, 4]$. Denote the resulting cycle by $C_{2}$. Then replace the edge $ww^{\prime}$ in $C_{2}$ with the path $wQ_{1}uP_{5}^{\prime}vQ_{2}w^{\prime}$. The resulting cycle is a desired cycle of $G$.


\noindent{\bf Case 2.} $k=l-2$.

If $ N_{D_{l}}(x_{l-1})\cup N_{D_{l}}(y_{l-1})\cup N_{D_{l-1}}(x_{l})\cup N_{D_{l-1}}(x_{l})=\emptyset $, then Lemma \ref{result 0}(iii) implie that each vertex in $ \left\{x_{l-1}, y_{l-1}, x_{l}, x_{l}\right\} $ is adjacent to all the vertices in at least $ l-3 $ components of $ D_{1}, \dots, D_{l-2} $. Now there exists $ j\in [1, l-2] $ such that $ x_{l-1}\sim D_{j} $ (resp. $D_{j+1}$) and $ y_{l-1}\sim D_{j+1} $ (resp. $D_{j}$), or $ x_{l-1}\sim D_{l-2}$ (resp. $D_{1}$) and $ y_{l-1}\sim D_{1} $ (resp. $D_{l-2}$). Now replace the path $ b_{j}y_{j}a_{j+1} $ or $a_{1}y_{l-1}b_{l-2}$ on $C_{1}$ with $ b_{j}x_{l-1}a_{l-1}P_{l-1}b_{l-1}y_{l-1}a_{j+1} $ or $a_{1}x_{l-1}a_{l-1}P_{l-1}b_{l-1}y_{l-1}b_{l-2}$, respectively. Note that the resulting cycle contains at least $l-1$ consecutive components. By a similar argument as Case 1, we can obtain a cycle containing all the vertices in $G-S$.

Now we assume $ N_{D_{l}}(x_{l-1})\cup N_{D_{l}}(y_{l-1})\cup N_{D_{l-1}}(x_{l})\cup N_{D_{l-1}}(x_{l})\neq \emptyset $. Without loss of generality, we assume $ y_{l-1}a_{l}\in E(G) $. Let $ P^{\prime}=x_{l-1}a_{l-1}P_{l-1}b_{l-1}y_{l-1}a_{l}P_{l}b_{l}y_{l} $. 	If there exists $ j\in [1, l-3] $ such that $ x_{l-1} \sim D_{j} $ (resp. $D_{j+1}$) and $ y_{l} \sim D_{j+1} $ (resp. $D_{j}$), or $ x_{l-1}\sim D_{1} $ (resp. $D_{l-2}$) and $y_{l}\sim D_{l-2}$ (resp. $D_{1}$), then replace the path $b_{j}y_{j}a_{j+1}$ or $b_{l-2}y_{l-2}a_{1}$ in $C_{1}$ with $b_{j}x_{l-1}P^{\prime}y_{l}a_{j+1}$ or $a_{1}x_{l-1}P^{\prime}y_{l}b_{l-2}$, respectively. Note that the resulting cycle is a desired cycle. Otherwise, Lemma \ref{result 0}(iii) implies that $ 5\leq l\leq 6 $, $x_{l-1}\sim D_{l}$ and  $ y_{l}\sim D_{l-1} $. Assume $D_{1}$ is nontrivial. If $\left|V(D_{1})\right|=2$, let $X=\bigcup_{i=2}^{l-2}\left\{a_{i}, b_{i}, y_{i}\right\}\cup\left\{y_{1}, y_{l-1}\right\} $, otherwise, let $X=\bigcup_{i=2}^{l-2}\left\{a_{i}, b_{i}, y_{i}\right\}\cup\left\{a_{1}, y_{1}, y_{l-1}\right\}$. Let 
\[
G_{1} = 
\begin{cases}
	G-X & \text{if } \left|V(D_{l-1})\right|=1 \text{~and~} \left|V(D_{l})\right|=1. \\
	G-X\cup \left\{b_{l-1}\right\}, & \text{if } \left|V(D_{l-1})\right|\geq 2 \text{~and~} \left|V(D_{l})\right|=1, \\
	G-X\cup \left\{a_{l} \right\}, & \text{if } \left|V(D_{l-1})\right|=1 \text{~and~} \left|V(D_{l})\right|\geq 2,\\
	G-X\cup \left\{b_{l-1}, a_{l} \right\}, & \text{if } \left|V(D_{l-1})\right|\geq 2 \text{~and~} \left|V(D_{l})\right|\geq 2, 
\end{cases}
\]
As $\kappa(G)\geq 16$ and $l\leq 6$, we have $\tau(G_{1})\geq 2$, so there exist two disjoint paths $Q_{1}$ and $Q_{2}$ in $G_{1}$ between $V(D_{l-1})\cup V(D_{l})$ and $D_{1}$, whose internal vertices belong to $V(G)\backslash\left(V(D_{1})\cup V(D_{l-1})\cup V(D_{l})\right)$. Let $V(Q_{1})\cap V(D_{1})=\left\{w\right\}$, $V(Q_{2})\cap V(D_{1})=\left\{w^{\prime}\right\}$, $V(Q_{1})\cap \left(V(D_{l-1})\cup V(D_{l})\right)=\left\{u\right\}$ and  $V(Q_{2})\cap \left(V(D_{l-1})\cup V(D_{l})\right)=\left\{v\right\}$.
 Let $P_{1}^{\prime}$ be a hamiltonian path of $D_{1}$ from $a_{1}$ to $b_{1}$ such that $ww^{\prime}\in E(P^{\prime})$, $P_{i}^{\prime}$ be a hamiltonian path of $D_{i}-V(Q_{1})\cup V(Q_{2})$ from $a_{i}$ to $b_{i}$ for each $i\in [2, l-2]$.  Now replace the path $P_{i}$ on $C_{1}$ with $P_{i}^{\prime}$ for each $i\in [1, l-2]$. Denote the resulting cycle by $C_{2}$. If $u$ and $v$ are in two distinct components, then let $P_{l-1}^{\prime}$ be a hamiltonian path of $D_{l-1}$ from $u$ to $b_{l-1}$ and $P_{l}^{\prime}$ be a hamiltonian path of $D_{l}$ from $a_{l}$ to $v$. 
Now replace the edge $ww^{\prime}$ in $C_{2}$ with $wQ_{1}uP^{\prime}_{l-1}b_{l-1}y_{l-1}a_{l}P^{\prime}_{l}vQ_{2}w^{\prime}$. The resulting cycle is a desired cycle of $G$. If $u$ and $v$ are in a common component, then assume $u, v\in V(D_{l-1})$. Let $P_{l-1}^{\prime\prime}$ be a hamiltonian path of $D_{l-1}$ from $u$ to $b_{l-1}$ such that $uv\in E(P^{\prime\prime}_{l-1})$. As $y_{l}\sim D_{l-1}$, let $Q=uy_{l}b_{l}P_{l}a_{l}y_{l-1}b_{l-1}P_{l-1}^{\prime}v_{l}$.  Relace the edge $ww^{\prime}$ in $C_{2}$ with $P^{\prime}$. The resulting cycle is a desired cycle of $G$.
\end{proof}

\begin{lem}\label{result 13}
	Let $ G $ be a $ (P_{3}\cup P_{1}) $-free graph with $ 0<\tau(G)\leq 1 $, and $ S $ be a tough set of $ G $. Then $ w\left(G-S\right)=\alpha(G) $.
\end{lem}
\begin{proof}
	It is easy to see that $ w\left(G-S\right)\leq \alpha(G) $. Suppose that $ w\left(G-S\right)<\alpha(G) $. Then by $ \tau(G)\leq 1 $, we have $ \left|S\right|\leq w\left(G-S\right)<\alpha(G) $. Let $ I $ be a maximum independent set of $ G $.
	By Lemma \ref{result 0}(ii), each component of $ G-S $ is complete, this implies that
	$ I\cap S\neq\emptyset $ and $I\cap V(G-S)\neq\emptyset$.
	Hence by Lemma \ref{result 0}(iii),  we have $ w\left(G-S\right)=2 $. Then  $ \left|S\right|\leq 2 $. Let $ D_{1} $ and $ D_{2} $ be two components of $ G-S $.  Let $ s_{1}\in S\cap I $ and $ x\in V(G-S)\cap I $. Without loss of generality, we assume $ x\in D_{1} $. Lemma \ref{result 0}(iii) implies that $ s_{1}\sim D_{2} $, so $V(D_{2})\cap I=\emptyset $. And since $ D_{1} $ is complete, we obtain $ \left(V(D_{1})\backslash\left\{x\right\}\right)\cap I=\emptyset $. Thus $ \emptyset\neq I\backslash\left\{s_{1}, x\right\}\subseteq S $. Let  $S\backslash\left\{s_{1}\right\}=\left\{s_{2}\right\} $. Then $ s_{2}\in I $ and $ N_{G}(s_{2})\cap\left\{a, s_{1}\right\} =\emptyset $, so by Lemma \ref{result 0}(iii), we have $ s_{2}\sim D_{2} $. Let $ y\in D_{2} $. Then the path $ys_{1}s_{2}$ and the vertex $ x$  induce a copy of $ P_{3}\cup P_{1} $, a contradiction.
\end{proof}

\begin{lem}[\cite{Bauer}]\label{result 2}
	Let $ t>0 $ be a real number and $ G $ be
	a  $ t $-tough graph on $ n\geq 3 $ vertices with $ \delta(G)>\frac{n}{t+1}-1 $. Then $ G $ is hamiltonian.
\end{lem}	
In the following, we define $N_{G}(x_{1}\dots x_{k})=\bigcup_{i=1}^{k} N_{G}(x_{i})$.

\begin{lem}\label{result 8}
Let $ G $ be a $ 15 $-tough $ (P_{3}\cup 3P_{1}) $-free graph of order $n$, and $ S $ be a proper cutset of $ G $ such that $ w(G-S)=3 $. If there exists a noncomplete component $ D_{1} $ of $ G-S $ containing at least $\frac{7}{8}n$ vertices, then $ G $ is hamiltonian, or there exists a clique $ Q_{1}$ in $ D_{1} $ such that $ \left|V(Q_{1})\right|-2\left|N_{G}(Q_{1})\right|\geq 2 $. 
\end{lem} 

\begin{proof}
	By $\tau(G)\geq 15$ and Lemma \ref{result 2}, we assume $30\leq \delta(G)\leq \frac{n}{16}-1$, so $ n\geq 496 $.
Let $ D_{2} $ and $D_{3}$ be the components of $ G-S$ other than $D_{1}$. Then 
\begin{equation}
	\left|S \right|=\left|V(G)\right|-\left|V(D_{1})\cup V(D_{2})\cup V(D_{3})\right|\leq n-\left(\frac{7}{8}n+2\right)=\frac{n}{8}-2.
\end{equation}
By Lemma \ref{result 0}(i), we have that both $ D_{2} $ and $ D_{3} $ are complete, and $ D_{1} $ is $ (P_{3}\cup P_{1}) $-free.

By simple calculations, we can obtain the following claim.

\noindent{\bf Claim 1.} If there exists a set $ A\subseteq V(D_{1})$ such that $ \left|A\right|\leq \frac{3n}{16}-3 $ and $\alpha(D_{1}-A)=1$, then $ D_{1}-A $ is a desired clique.

Let $ S_{1}=\left\{x\in S: \left|N_{D_{1}}(x)\right|>\frac{n}{16}-1\right\} $ and $S_{2}=S\backslash S_{1}$. By Claim 1, we may assume each vertex $w\in S_{2} $ satisfies $ \alpha(D_{1}-N_{D_{1}}(w))\geq 2 $, call this {\bf Assumption 1}. Then by the $ (P_{3}\cup 3P_{1}) $-freeness of $ G $, we obtain the following claim.

\noindent{\bf Claim 2.} Each vertex in $ S_{2} $ is adjacent to all the vertices in $ D_{2} $ or $D_{3}$.

If $ \alpha(G[S_{2}])\geq 2 $, then let  $ p=\min\left\{\alpha(G[S_{2}]), 3 \right\} $ and $ q\in [2, p] $. By Claim 1, we assume that any $q$ independent vertices $w_{1}, \dots, w_{q}$ in $S_{2}$  satisfy $ \alpha\left(D_{1}-\bigcup_{i=1}^{q}N_{D_{1}}(w_{i})\right)\geq 2$, call this {\bf Assumption 2}. 		
Let $ G_{1}=G[V(D_{2})\cup V(D_{3})\cup S_{2}] $.

\noindent{\bf Claim 3.} 
If $ \alpha(G[S_{2}])\geq 2 $, then $ \alpha(G_{1})=\alpha(G[S_{2}]) $. 

\begin{proof}
Since $S_{2}\subseteq V(G_{1})$, we have $\alpha(G_{1})\geq \alpha(G[S_{2}])$. Suppose to the contrary that $ \alpha(G_{1})> \alpha(G[S_{2}])\geq 2 $. Let $ I $ be a maximum independent set of $ G_{1} $. Then $ I\cap S_{2}\neq\emptyset $ and $I\cap \left(V(D_{2})\cup V(D_{3})\right)\neq\emptyset $. Since $ D_{2} $ and $ D_{3} $ are complete and by Claim 2, we have $ \left|I\cap \left(V(D_{2})\cup V(D_{3})\right)\right|=1 $, so $\left|I\cap S_{2}\right|\geq 2$. Let $ I\cap \left(V(D_{2})\cup V(D_{3})\right)=\left\{u\right\} $ and $\left\{w_{1}, w_{2}\right\}\subseteq  I\cap S_{2}$. Then $\left\{w_{1}, w_{2}\right\}\cap N_{G}(u)=\emptyset$. Without loss of generality, assume that $ u\in V(D_{2})$. Let  $ v\in V(D_{3}) $.  By Claim 2, we have  $ \left\{w_{1}, w_{2}\right\}\subseteq N_{G}(v)$. By assumption 2 that $\alpha\left(D_{1}-N_{D_{1}}(w_{1}w_{2})\right)\geq 2$, $ N_{D_{1}}(v)=\emptyset $ and $\left\{w_{1}, w_{2}, v\right\}\cap N_{G}(u)=\emptyset$, we have $ \alpha(G-N_{G}(w_{1}vw_{2}))\geq 3 $, a contradiction to the $(P_{3}\cup 3P_{1})$-freeness of $G$.
\end{proof} 

\noindent{\bf Claim  4.}	If $ \tau(D_{1})>1 $, $ \tau\left(G_{1}\right)>1 $ and $ G_{1} $ is $ (P_{3}\cup P_{1}) $-free, then $ G $ is hamiltonian. 

\begin{proof}
	By $ \kappa(G)\geq 30 $, there exist two disjoint paths $ P_{1} $ and $ P_{2} $ in $ G $ between $ V(G_{1}) $ and $ V(D_{1}) $ with all the internal vertices in $S_{1}$. For each $ i\in \left\{1, 2 \right\}$, let $ x_{i}$ and $ y_{i}$ be two ends of $ P_{i} $, where $ x_{i}\in V(G_{1})$ and $y_{i}\in V(D_{1}) $.  By Lemma \ref{result 5}, let $ Q_{1} $  be a hamiltonian path of $ G_{1} $ with two ends $ x_{1} $ and $ x_{2} $, and  $ Q_{2} $  be a hamiltonian path of $ D_{1} $ with two ends $ y_{1} $ and $ y_{2} $. Note that $C=x_{1}Q_{1}x_{2}P_{2}y_{2}Q_{2}y_{1}P_{1}x_{1}$ is a cycle containing all the vertices in $ V(G_{1})\cup V(D_{1}) $, so $V(G)\backslash V(C)\subseteq S_{1}$. Applying Lemma \ref{result 4} on the vertices in $V(G)\backslash V(C)$, we can obtain a hamiltonian cycle of $G$. 
\end{proof}

Let $ T $  be a tough set of $ D_{1} $. By Lemma \ref{result 0}(ii), each component of $D_{1}-T$ is complete.

\noindent{\bf Claim 5.} Assume $ 0<\tau(D_{1})\leq 1 $.  Then $ \left|T\right|\leq w\left(D_{1}- T\right)< \frac{n}{112}-2$.  If there exists a vertex $ w\in S $ such that $0< \left|N_{D_{1}-T}(w)\right|\leq \frac{n}{16}-1 $, or  $ D_{1}-T $ has at most two nontrivial components, then there exists a desired clique.

\begin{proof}
	As $ \tau(D_{1})\leq 1 $, we have  $\left|T\right|\leq w\left(D_{1}-T\right) $. Note that $ T\cup S $ is a cutset of $ G $. So
	$$ 15\leq \frac{\left|T\cup S\right|}{w\left(G-(T\cup S)\right)}\leq \frac{w\left(D_{1}- T\right)+\frac{n}{8}-2}{w\left(D_{1}- T\right)+2}$$ and then $ \left|T\right|\leq w\left(D_{1}- T\right)< \frac{n}{112}-2$. If $ w\left(D_{1}-T\right)=2 $, then since each component of $D_{1}-T$ is complete and $\left|V(D_{1})\right|\geq \frac{7n}{8}$, a largest component of $D_{1}-T$ is a desired clique. Thus we assume $w\left(D_{1}-T\right)\geq 3 $.

	If there exists a vertex $ w\in S $ such that $0< \left|N_{D_{1}-T}(w)\right|\leq \frac{n}{16}-1 $, then by Lemma \ref{result 0}(iii), $ w $ is adjacent to all the vertices in at least $ w(D_{1}-T)-2 $ components of $ D_{1}-T $, so  there exist $ w\left(D_{1}-T\right)-2 $ components in $D_{1}-T$ containing at most $ \frac{n}{16}-1 $ vertices. If there exist at most
	2 nontrivial components in $ D_{1}-T $, then as $w\left(D_{1}- T\right)<\frac{n}{112}-2<\frac{n}{16}+1$, we have that $ w(D_{1}-T)-2 $ trivial components of $ D_{1}-T $ contain less than $ \frac{n}{16}-1 $ vertices.
	Thus in each case, there exists a complete component $ Q_{1} $ of $ D_{1}-T $ such that 
	$$\left|V(Q_{1})\right|\geq\frac{\left|V(D_{1})\right|-\left|T\right|-\left(\frac{n}{16}-1\right)}{2}> \frac{\frac{7}{8}n-(\frac{n}{112}-2)-(\frac{n}{16}-1)}{2}= \frac{45n}{112}+\frac{3}{2}.$$ 
	Note that $ N_{G}(Q_{1})\subseteq S\cup T $.
	So
	$$ \left|N_{G}(Q_{1})\right|<
	\frac{n}{8}-2+\left(\frac{n}{112}-2\right)=\frac{15n}{112}-4.$$
	Hence $ \left|V(Q_{1})\right|-2\left|N_{G}(Q_{1})\right|> \frac{15}{112}n+\frac{19}{2}>2 $. 
\end{proof}

In the following, we consider two cases according to  $\alpha(G[S_{2}])$.		

\noindent{\bf Case 1.}	$ \alpha(G[S_{2}])\leq 2 $.

By Claims 2 and 3, we have $ \alpha(G_{1})=2 $, so $ G_{1} $ is $ (P_{3}\cup P_{1}) $-free. Let $ S_{21}=N_{S_{2}}(D_{2}) $, $ S_{22}=N_{S_{2}}(D_{3}) $, $ G_{11}=G[S_{21}\cup V(D_{2})] $ and $G_{12}=G[S_{22}\cup V(D_{3})]$.

\noindent{\bf Case 1.1.} $ 0<\tau(D_{1})\leq 1 $.

By Claim 5, we have $ \left|T\right|<\frac{n}{16}-1 $. So by the definition of $ S_{1} $, each vertex in $S_{1}$ has a neighbor in $D_{1}-T$. Again by the defintion of $ S_{2} $ and Claim 5, we assume that $ S_{2}\cap N_{G}(D_{1}-T)=\emptyset $, $D_{1}-T$ has at least three nontrivial components, and each vertex in $S_{1}$ has more than $\frac{n}{16}-1$ neighbor in $D_{1}-T$. Then $ N_{D_{1}}(S_{2})\subseteq T $.  Lemma \ref{result 0}(iii) implies that $ V(G_{11})\cap V(G_{12})=\emptyset $ and $E_{G}(G_{11}, G_{12})=\emptyset  $. Thus $ w\left(G-\left(S_{1}\cup T\right)\right)\geq 5 $. By the $(P_{3}\cup P_{1})$-freeness of $D_{1}$ and Lemma \ref{result 0}(iii), each vertex in $ T $ has  $\left|V\left(D_{1}-T\right)\right|>\frac{n}{16}-1 $ neighbors in $ D_{1}-T $. Therefore, by Lemma \ref{result 11}, $ G $ is hamiltonian.

\noindent{\bf Case 1.2.} $ \tau(D_{1})>1 $. 

By Claim 4, we assume $ \tau(G_{1})\leq 1 $. If $G_{1}$ is connected, then let $T_{1}$ be a tough set of $G_{1}$. So $ 1\leq \left|T_{1}\right| \leq w(G_{1}-T_{1})\leq  \alpha(G_{1})= 2 $.
Let $X=S_{1}\cup V(D_{1})$.
As $\delta(G)\geq \kappa(G)\geq 30 $ and each component of $ G_{1}-T_{1} $ is complete by Lemma \ref{result 0}(ii), there exists a hamiltonian path $ P_{1} $ of $G_{1}$ with two ends $a$ and $b$ such that  
$ \left|N_{X}(a)\right|\geq 1 $, $ \left|N_{X}(b)\right|\geq 1 $ and $ \left|N_{X}(ab) \right|\geq 2$.    Let $ w_{1}\in N_{G}(a)\cap X $ and $ w_{2}\in N_{G}(b)\cap X $ be two distinct vertices. By Lemma  \ref{result 5} and the definition of $S_{1}$, there exists a path $ P_{2} $ between $w_{1}$ and $w_{2}$ that covers all the vertices in $ D_{1} $ and all the internal vertices are in $X$. Then $ C=w_{1}aP_{1}bw_{2}P_{2}w_{1} $ is cycle such that $ V(D_{1})\cup V(G_{1})\subseteq V(C) $. Applying Lemma \ref{result 4}  on the vertices of $S_{1}$, $ G $ is hamiltonian.

Thus we assume $ G_{1} $ is not connected. Claim 2 implies that $ V(G_{11})\cap V(G_{12})=\emptyset $ and $E_{G}(G_{11}, G_{12})=\emptyset  $. By Lemma \ref{result 0}(ii),  we have that $ G_{11} $ and $ G_{12}$ are complete. 

By $ V(G_{11})\cap V(G_{12})=\emptyset $, if $ S_{2}\neq \emptyset $, then as $S$ is a propoer cutset of $G$, we have that each vertex in $ S_{2} $ has a neighbor in $ D_{1} $. 
For each $ i\in \left\{1, 2\right\} $, if $ S_{2i}\neq\emptyset $, then let $ s_{i}\in S_{2i} $, otherwise, let $ s_{i}\in N_{S_{1}}(G_{1i}) $. Let $ w_{i}\in N_{D_{1}}(s_{i}) $ and $ e_{i}=s_{i}w_{i}$. 
Let $ P^{\prime} $ be the shortest path in $ D_{1} $ between $w_{1}$ and $w_{2}$. Then since $ D_{1} $ is $ (P_{3}\cup P_{1}) $-free, we have $ \left|V(P^{\prime})\right|\leq 4 $.  Let $ D_{11}=D_{1}-V(P^{\prime}) $. Then $ D_{11} $ is $ (P_{3}\cup P_{1}) $-free. Note that $N_{G}(G_{11})\cup N_{G}(G_{12})\subseteq S_{1}\cup V(D_{1}) $.
By $ \kappa(G-\left(V(P^{\prime})\cup\left\{s_{1}, s_{2}\right\}\right))\geq 24 $ and the definition of $ S_{1} $, there exist at least 3 disjoint edges between $ \left(S_{2i}\cup N_{S_{1}}(G_{1i})\right)\backslash\left\{s_{i}\right\} $ and $ V(D_{11}) $ for each $ i\in [1, 2] $. 
Let $ e_{3}=s_{3}w_{3} $ and $ e_{4}=s_{4}w_{4}$ be two disjoint edges such that $ s_{3}\in \left(S_{21}\cup N_{S_{1}}(G_{11})\right)\backslash\left\{s_{1}\right\} $, $ s_{4}\in \left(S_{22}\cup N_{S_{1}}(G_{12})\right)\backslash\left\{s_{2}\right\} $ and $\left\{ w_{3}, w_{4}\right\}\subseteq V(D_{11}) $. Then since $ G_{11} $ and $ G_{12} $ are complete, there exists a path $ P $ from $s_{3}$ to $s_{4}$  such that  $ V(P)=V(G_{1})\cup V(P^{\prime})\cup\left\{s_{3}, s_{4}\right\} $.

If $ \tau(D_{11})>1 $, then by Lemma \ref{result 5}, there is a hamiltonian path $ P_{2} $ of $ D_{11} $ from $ w_{3} $ to $ w_{4} $. 
Then $C= w_{3}P_{2}w_{4}s_{4}Ps_{3}w_{3} $ is a cycle such that $ V(D_{1})\cup V(G_{1})\subseteq V(C) $. Applying Lemma \ref{result 4} resursively on vertices of $ S_{1} $, we can obtain a hamiltonian cycle of $ G $. 	Hence we assume $ \tau(D_{11})\leq 1 $.
If $ D_{11} $ is connected, then let $ T_{2} $ be a tough
set of $ D_{11} $, otherwise, we define $ T_{2}=\emptyset $. 
As $\tau(D_{11})\leq 1$ and $\left|V(P^{\prime})\right|\leq 4$, we have $\left|T_{2}\cup V(P^{\prime})\right|\leq w\left(D_{11}-T_{2}\right)+4$. Replace the role of $T$ in Case 1.1 with $T_{2}\cup V(P^{\prime})$. By a similar argument as Case 1.1 and by $n\geq 496$, we know that the lemma holds.

\noindent{\bf Case 2.} $ \alpha(G[S_{2}])\geq 3 $.

\noindent{\bf Claim 6.} For any three pairwise nonadjacent vertices $ a, b, c\in S_{2} $, $D_{1}-N_{D_{1}}(abc) $ has exactly 2 complete components.
\begin{proof}
	By Claim 2, assume $\left\{a, b\right\}\sim D_{2}$. Let $u\in D_{2}$. Consider the path $aub$ and the vertices in $D_{1}-N_{D_{1}}(abc)$. By assumption 2, the $(P_{3}\cup 3P_{1})$-freeness of  $G$ and $N_{D_{1}}(u)=\emptyset$, we have that $D_{1}-N_{D_{1}}(abc)$ has exactly 2 complete components. 
\end{proof}

If there exist three pairwise nonadjacent vertices $ a_{1}, a_{2}, a_{3}\in S_{2} $ such that one component of  $D_{1}-N_{D_{1}}(a_{1}a_{2}a_{3})$ containing at most $\frac{n}{16}-1$ vertices, then by Claim 6, the other component is a desired clique.		
So we assume for any three independent vertices $ a, b$ and $c $ in $ S_{2}$, each component of  $D_{1}-N_{D_{1}}(abc)$ contains more than $\frac{n}{16}-1$ vertices, call this {\bf Assumption 3.}
Then since each vertex in $G_{1}$ has at most $\frac{n}{16}-1$ neighbors in $D_{1}$ and $ G $ is $ (P_{3}\cup 3P_{1}) $-free, we have that $ G_{1} $ is $ \left(P_{3}\cup P_{1}\right) $-free.

\noindent{\bf Claim 7.}	Let $a$ be a vertex in $S_{2}$ contained in an independent set of order 3 in $G[S_{2}]$. Then $ V(D_{2})\cup V(D_{3})\subseteq N_{G}(a) $. 

\begin{proof}
	Suppose to the contrary that there exists a vertex $ u\in V(D_{2})\cup V(D_{3}) $ such that $ au\notin E(G) $. Let $ b $ and $ c $ be two nonadjacent vertices in $ S_{2} $ such that $ \left\{b, c\right\}\cap N_{G}(a)=\emptyset $. Since  $ G_{1} $ is $ \left(P_{3}\cup P_{1}\right) $-free, we have $ ub\notin E(G) $ or $ uc\notin E(G) $. Assume  $ ub\notin E(G) $ and $ u\in V(D_{2}) $. Let $ v\in V(D_{3}) $. Then by Claim 2, we have $ \left\{a, b\right\}\subseteq  N_{G}(v) $. By Claim 6 and  $ u\notin N_{G}(avb) $, we can obtain $ \alpha (G-N_{G}(avb))\geq 3 $, which contradicts the $ \left(P_{3}\cup 3P_{1}\right) $-freeness of $ G $.
\end{proof}

By Claim 7, there exists a vertex in $ s\in S_{2} $ such that $ V(D_{2})\cup V(D_{3})\subseteq N_{G}(s) $. Then $ D_{1}-T $ has at most 2 nontrivial components, or $0<\left|N_{D_{1}-T}(s)\right|\leq \frac{n}{16}-1$ by Lemma \ref{result 0}(iii). If $ \tau(D_{1})\leq 1 $, then by Claim 5, there exists a clique $Q_{1}$ such that $\left|V(Q_{1})\right|-2\left|N_{G}(Q_{1})\right|\geq 2$.
Hence we assume $ \tau(D_{1})> 1$. And by Claim 4, we assume $ 0<\tau(G_{1})\leq 1$. Let $T_{1}$ be a tough set of $G_{1}$.  By Lemma \ref{result 13} and Claim 3, we have $w\left(G_{1}-T_{1}\right)=\alpha(G_{1})= \alpha (G[S_{2}])\geq 3$.
By Lemma \ref{result 0}(iii), each vertex in $ T_{1} $ is adjacent to all the vertices in $ G_{1}-T_{1} $, this implies that $ V(D_{2})\cup V(D_{3})\subseteq T_{1} $ or $V(D_{2})\cup V(D_{3})\subseteq V(G_{1}-T_{1})$. 
Moreover, by Claim 2 and $w\left(G_{1}-T_{1}\right)\geq 3$,  we have $ V(D_{2})\cup V(D_{3})\subseteq T_{1} $, so $V(G_{1}-T_{1})\subseteq S_{2}$. 

Choose an arbitrary vertex from each component of $G_{1}-T_{1}$. Denote the set of these vertices by $U$. Then $U$ is independent. Let $\ell=\left|U\right|$, $W=N_{D_{1}}(G_{1}-T_{1})$, and  $W_{0}\subseteq W$ be the set of vertices not adjacent to any vertex in $D_{1}-W$. Then $\ell=w\left(G_{1}-T_{1}\right)\geq 3$.

\noindent{\bf Claim 8.}	 For any three pairewise nonadjacent vertices $a, b, c$ in $U$, we have $W=N_{D_{1}}(abc)$.  Consequently, each vertex in $W$ is adjacent to all the vertices in at least $\mathcal{l}-2$ component of $G_{1}-T_{1}$. Moreover, each vertex in $W_{0}$ is adjacent to all the vertices in $G_{1}-T_{1}$.

\begin{proof}
	By Claim 6, let $x$ and $y$ be two vertices in two distinct  components of $D_{1}-N_{G}(abc)$, respectively. Let $u\in T_{1}$. By Lemma \ref{result 0}(iii), we have $V(G_{1}-T_{1})\subseteq N_{G}(u)$.
	
	It is sufficient to prove that $W\subseteq N_{G}(abc)$. Suppose to the contrary that there exists a vertex $z\in W$ with $z\in W\backslash N_{G}(abc)$.  
	Consider the path $aub$ and three vertices $ x, y, z$. Since $G$ is $(P_{3}\cup 3P_{1})$-free, we have $xz\in E(G)$ or $yz\in E(G)$. Assume $xz\in E(G)$. Let $d\in N_{G_{1}-T_{1}}(z)$. Then $\left|\left\{a, b, c\right\}\cap N_{G}(d)\right|\leq 1$. Assume $\left\{a, b\right\}\cap N_{G}(d)=\emptyset$.  Cosider the path $xzd$ and the vertices $a, b, y$. Since $G$ is $(P_{3}\cup 3P_{1})$-free, we have $yz\in E(G)$. Now the path $xzy$ and the vertices $a, b, c$ induce a copy of $P_{3}\cup 3P_{1}$, a contradiction.

Suppose to the contrary that there exists a vertex $w_{0}\in W_{0}$ with $V(G_{1}-T_{1})\backslash N_{G}(w_{0})\neq\emptyset$. Let $a_{1}\in V(G_{1}-T_{1})\backslash N_{G}(w_{0})$.  If $\left|N_{U}(w_{0})\right|\geq 2$, then let $ b_{1}, c_{1}\in N_{U}(w_{0}) $.  Now the path $b_{1}w_{0}c_{1}$ and three vertices $a_{1}, x, y$ induce a copy of $P_{3}\cup 3P_{1}$, a contradiction. If $\left|N_{U}(w_{0})\right|=1$, then as $\ell\geq 3$, let $b_{2}, c_{2}\in U\backslash N_{G}(w_{0}) $. Consider the path $b_{2}uc_{2}$ and three vertices $w_{0}, x, y$ induce a copy of $P_{3}\cup 3P_{1}$, a contradiction.
\end{proof}

By Claims 6 and 8, $D_{1}-W$ has exactly two complete components. If $\left|W\right|\leq \frac{n}{16}-1 $, then there exists a component $D$ of $D_{1}-W$ such that $\left|V(D)\right|-2\left|N_{G}(D)\right|\geq 2$. So we assume $\left|W\right|>\frac{n}{16}-1 $.  
 Suppose that any two vertices $x, y\in W$ satisfy $\left|N_{U}(xy)\right|\leq \ell-2$. Then Claim 8 implies that all the vertices in $W$ are only adjacent to  $\ell-2$ common vertices in $U$. Let $a, b\in U\backslash N_{G}(W)$ and  $c\in U\backslash\left\{a, b\right\}$. Then by Claim 8 and  $ V(G_{1}-T_{1})\subseteq S_{2} $, we have $\left|W\right|=\left|N_{D_{1}}(abc)\right|=\left|N_{G}(c)\right|\leq \frac{n}{16}-1 $, a contradiction.  So there exist two vertices $y_{1}, y_{2}\in W$ with 
$\left|N_{U}(y_{1}y_{2})\right|\geq \ell-1$.

Since $ W\cup 
V(D_{2}) \cup V(D_{3})\cup (S\backslash U)$ is a cutset of $ G $, we have  $$  \frac{\left|W\right|+\left|V(D_{2}) \cup V(D_{3})\cup (S\backslash U)\right|}{\ell+2}\geq 15.$$ By Claim 8 and the definition of $S_{2}$, we have $ \left|W\right|\leq \frac{3n}{16}-3 $. As $\left|V(D_{2}) \cup V(D_{3})\cup S\right|= n-\left|V(D_{1})\right|\leq \frac{n}{8}$, we have $ \mathcal{l}\leq \frac{5n}{256}-1 $. If $ \left|W\right|\leq \ell-2\leq \frac{5n}{256}-3 $, then there exists a clique $ Q_{1} $ such that $$\left|V(Q_{1})\right|\geq \frac{\left|V(D_{1})\right|-\left|W\right|}{2}\geq \frac{\frac{7}{8}n-\left(\frac{5n}{256}-3\right)}{2}=\frac{219}{512}n+\frac{3}{2}.$$ Since $ N_{G}(Q_{1})\subseteq W\cup S $, we have $ 2\left|N_{G}(Q_{1})\right|\leq 2\left(\frac{5n}{256}-3+ \frac{n}{8}-2\right)=\frac{37}{128}n-10 $. Thus $ \left|V(Q_{1})\right|-2\left|N_{G}(Q_{1})\right|\geq 2 $. Hence we assume  $\left|W\right|\geq \max\left\{\ell-1, \frac{n}{16}-1\right\} $.

  Let $k=\max\left\{\ell-\left|T_{1}\right|+1, 2\right\}$. If $\left|W_{0}\right|\geq k $, then let $ W^{\prime}\subseteq W_{0} $ be a set of size $k$.  If $0\leq \left|W_{0}\right|<k $, then let $ W^{\prime}\subseteq W_{0} $ be a set such that $ W_{0}\subseteq W $, $ \left|W^{\prime}\right|=k $, and there exist 2 vertices in $W^{\prime}$ adjacent to at least $\ell-1$ vertices in $U$. In each case, there exist 2 vertices $y_{1}, y_{2}\in W^{\prime}$ such that $\left|N_{U}(y_{1}y_{2})\right|\geq \ell-1$.
Let $W_{1}=W^{\prime}-\left\{y_{1}, y_{2}\right\}$.  

\noindent{\bf Claim 9.} There exists a hamiltonian path of $G\left[V(G_{1})\cup W^{\prime}\right]$ with two ends $y_{1}$ and $y_{2}$.

\begin{proof}
	Let $B_{1}, B_{ 2}\dots, B_{\ell}$ be the components of $G_{1}-T_{1}$. If $k=2$, then $W_{1}=\emptyset$ and $\mathcal{l}-1\leq \left|T_{1}\right|\leq l$. 
	Let $u\in V(D_{2})$ and $v\in V(D_{3})$. By Claim 2 and as $D_{2}$ and $D_{3}$ are complete, each vertex in $G_{1}\backslash\left\{u, v\right\}$ is adjacent to $u$ or $v$. Moreover, as $\left\{u, v\right\}\subseteq V(D_{2})\cup V(D_{3})\subseteq T_{1}$. there exist exactly $\ell-1$ disjoint paths in $G[T_{1}]$ covering all the vertices in $T_{1}$.	
	Since $\left|N_{U}(y_{1}y_{2})\right|\geq \ell-1$, we assume $N_{B_{ 1}}(y_{1})\neq\emptyset$ and $N_{B_{ \ell}}(y_{2})\neq \emptyset$. Let $a\in N_{B_{ 1}}(y_{1})$ and $b\in N_{B_{ \ell}}(y_{2})$.
	Since each component of $G_{1}-T_{1}$ is complete and each vertex in $T_{1}$ is adjacent to all the vertices in $G_{1}-T_{1}$, there exists a hamiltonian path $P$ of $G_{1}$ between $a$ and $b$. Then $y_{1}aPby_{2}$ is a desired path. 
	
	Thus we assume $k\geq 3$. Then $k=\ell-\left|T_{1}\right|+1$.  Let $W_{1}=\left\{y_{3}, \dots, y_{k}\right\}$. As $\left|T_{1}\right|\geq \left|V(D_{2})\cup V(D_{3})\right|\geq 2$, we have $\mathcal{l}\geq 4$. If $k=3$, then by Claim 8, assume $y_{3}\sim \left\{B_{1}, B_{ 2}\right\}$. If $k\geq 4$, then by Claim 8, assume $\left\{y_{i+2}, y_{i+3}\right\}\sim B_{i}$ for each $i\in [1, k-3]$,  $y_{3}\sim B_{k-2}$ and $y_{k}\sim B_{k-1}$. Since $\left|N_{U}(y_{1}y_{2})\right|\geq \ell-1$, we assume $N_{B_{k-2}}(y_{1})\neq\emptyset$ and $N_{B_{k}}(y_{2})\neq \emptyset$. Since each component of $G_{1}-T_{1}$ is complete, there exsits a hamiltonian path $P_{1}$ of $G\left[\left(\bigcup_{i=1}^{k-2} V(B_{i})\cup W_{1}\cup\left\{y_{1}\right\}\right)\right]$ from $ y_{1} $ to $y_{k}$. Let $a\in N_{B_{k-2}}(y_{k})$ and $b\in N_{B_{k}}(y_{2})$. Since each vertex in $T_{1}$ is adjacent to all the vertices in $G_{1}-T_{1}$, there exists a hamiltonian path of $G\left[\left(\bigcup_{i=k-1}^{\mathcal{l}} V(B_{i})\cup T_{1}\right)\right]$ from $a$ to $b$.  Thus $y_{1}P_{1}y_{k}aP_{2}by_{2}$ is a hamiltonian path of $G[V(G_{1})\cup W^{\prime}]$.
\end{proof}


Let $ H_{1}=D_{1}-W_{1} $.
If $\tau(H_{1})>1$, then by Lemma \ref{result 5},  $ H_{1} $ has a hamiltonian path from $y_{1}$ to $y_{2}$. By Claim 9, we can obtain a cycle covering all the vertices in $V(D_{1})\cup V(G_{1})$. Applying Lemma \ref{result 11} on the remaining vertices, we can obtain a hamiltonian cycle of $G$. 
Hence we assume $ \tau(H_{1})\leq 1 $. Let $ T_{2} $ be  a tough set of  $H_{1}$. Note that $V\left(H_{1}- T_{2}\right)\cap V(D_{1}-W)\neq\emptyset$. Otherwise, we have $V(D_{1}-W)\subseteq T_{2}$ and $V(H_{1}-T_{2})\subseteq W$. As  $ \left|V(D_{1})\right|\geq \frac{7n}{8} $ and $\left|W\right|\leq \frac{3n}{16}-3$, we have $\frac{\left|T_{2}\right|}{w\left(H_{1}-T_{2}\right)}>\frac{\left|V(D_{1}-W)\right|}{\left|W\right|}>1$, a contradiction. Suppose that $w\left(H_{1}- T_{2}\right)\geq 3$. Then by Lemma \ref{result 0}(iii), every vertex in $T_{2}$ is adjacent to all the vertices in $H_{1}-T_{2}$. Since $V(H_{1}-T_{2})\cap V(D_{1}-W)\neq\emptyset$, $w\left(D_{1}-W\right)=2$ and  each component of $D_{1}-W$ is complete, we have that there exist two components $B_{1}$ and $B_{2}$ of $H_{1}-T_{2}$ containing all the verictes in two components of $ D_{1}-W $, respectively. 
 This implies that $ W_{0}\cap T_{2}=\emptyset$, so $V(H_{1})\cap W_{0}\subseteq H_{1}-T_{2}$. Since each vertex in $W-W_{0} $ is adjacent to all the vertices in one components of $D_{1}-W$ by Lemma \ref{result 0}(ii), we have $V(H_{1})\cap W-W_{0}\subseteq V(H_{1}-T_{2})$. Combining with  $w\left(H_{1}- T_{2}\right)\geq 3$, we have $\emptyset\neq V(H_{1})\cap W_{0}\subseteq V\left(H_{1}-T_{2}\right)$. By the constrcution of $W^{\prime}$, we have $\left|W_{0}\right|\geq k$ and $W_{1}\subsetneqq W^{\prime}\subseteq W_{0}$. Note that $T_{2}\cup W_{1}$ is a cutset of $D_{1}$. By the defintion of $W_{0}$,  we have $V(B_{1})-W_{1}\neq\emptyset$ and  $V(B_{2})-W_{1}\neq\emptyset$. Lemma 2.1(iii) implies that each vertex has neighbors in at most one component of $ D_{1}-\left(W_{1}\cup T_{2}\right) $. Again, by Lemma \ref{result 0}(iii), we have 
 $w\left(D_{1}-T_{2}\right)\geq w\left(H_{1}- T_{2}\right)\geq \left|T_{2}\right|$, hence $\tau(D_{1})<1$, which contradicts our assumption that $\tau(D_{1})>1$.
  Thus $w\left(H_{1}- T_{2}\right)=2$.
As $ \tau(H_{1})\leq 1 $, we have $ \left|T_{2}\right|\leq 2 $. Since $\left|V(D_{1})\right|\geq \frac{7n}{8}$, $ \ell\leq \frac{5n}{256}-1 $ and $\left|W_{1}\right|\leq k-2\leq \max \left\{\ell-3, 0\right\} $, there exists a complete component $ Q_{1} $ of  $ H_{1}-T_{2}$ such that $ \left|V(Q_{1})\right|-2\left|N_{G}(Q_{1})\right|\geq 2 $.
\end{proof}

The following lemmas play a crucial role in the proof of Theorem \ref{result 12}.

\begin{lem}[\cite{V}]\label{result 7}
Let $ G $ be a $ k $-connected graph on at least $  3 $ vertices.

\begin{enumerate}[\rm(i)]	
	\item If $ \kappa(G)\geq \alpha(G)-1 $, then $ G $ has a hamiltonian path.
	
	\item If $\kappa (G)\geq \alpha(G)$, then $ G $ is hamiltonian.
	
	\item If $ \kappa(G)\geq \alpha(G)+1 $, then $ G $ is hamiltonian-connected. 
\end{enumerate}	
\end{lem}

\begin{lem}[\cite{Shan2}]\label{result 3}
Let $ t>0 $ be a real number and $ G $ be
a  $ t $-tough graph on $ n\geq 3 $ vertices. If the degree sum of any two nonadjacent vertices of $ G $ is greater than $ \frac{2n}{t+1}-2 $, then  $ G $ is hamiltonian.
\end{lem}	
Let $x$ and $y$ be two distinct vertices in $P$ such that $xy\notin E(P)$. Denote the path $xPy-\left\{x, y\right\}$ by $P(x, y)$.

\noindent{\bf Proof of Theorem \ref{result 12}.}
Let $n=\left|V(G)\right|$. By Lemma \ref{result 3}, we assume that there exist two nonadjacent vertices $u, v\in G $ such that  $d_{G}(u)+d_{G}(v)\leq \frac{n}{8}-2 $.     Let $$ N_{u}=\left\{x\in N_{G}(uv): N_{G}(x)\subseteq \left\{u\right\}\cup N_{G}(uv)\right\},$$$$ 
N_{v}=\left\{x\in N_{G}(uv)\backslash N_{u} : N_{G}(x)\subseteq \left\{v\right\}\cup\left(N_{G}(uv)\backslash N_{u}\right) \right\} \text{~and~} S=N_{G}(uv)\backslash\left(N_{u}\cup N_{v}\right).$$ 
Note that $S$ is a proper cutset, $\left|S\right|\leq \left|N_{G}(uv)\right|\leq d_{G}(u)+d_{G}(v)\leq \frac{n}{8}-2$ and $w\left(G-S\right)=w\left(G-N_{G}(uv)\right)\geq 3$. 
If any two components $D_{i}$ and $D_{j}$ of $G-N_{G}(uv)$ satisfy that $$ \left|V(G)\backslash \left(N_{G}(uv)\cup V(D_{i})\cup V(D_{j})\right)\right|>\frac{n}{16}-1 ,$$ then as $\tau(G)\geq 15$, we have that there exist at least 3 nontrivial components in $G-N_{G}(uv)$, so $w\left(G-S\right)=w\left(G-N_{G}(uv)\right)\geq 5$. By Lemma \ref{result 0}(iii),  each vertex in $S$ has more than $\frac{n}{16}-1$ neighbors in $ G-S $. Thus by Lemma \ref{result 11}, $G$ is hamiltonian. In the following, we assume that there exist two components $  D^{\prime}, D^{\prime\prime} $ of $ G-N_{G}(uv) $ such that $$ \left|V(G)\backslash \left(N_{G}(uv)\cup V(D^{\prime})\cup V(D^{\prime\prime})\right)\right|\leq \frac{n}{16}-1 .$$

Let $D_{1}$ be the component of $G-S$ containing all the vertices in a largest component of $G-N_{G}(uv)$. Then $ N_{G}(D_{1})\subseteq S $ and $$ \lvert V(D_{1})\rvert\geq \max\left\{ \left|V(D^{\prime})\right|, \left|V(D^{\prime\prime})\right|\right\}\geq \frac{\left|V(G)\right|-\left|N_{G}(uv)\right|-\left(\frac{n}{16}-1\right)}{2}\geq \frac{13n}{32}+\frac{3}{2}. $$ We claim that $ G $ is hamiltonian, or there exists a clique $ Q_{1} $ in $D_{1}$ such that $ \left|V(Q_{1})\right|-2\left|N_{G}(Q_{1})\right|\geq 2 $. If $ D_{1} $ is complete, then $ D_{1} $ is a desired clique since $ \left|V(D_{1})\right|-2\left|N_{G}(D_{1})\right|\geq \left|V(D_{1})\right|-2\left|S\right|\geq  \frac{13n}{32}+\frac{3}{2}-2\left(\frac{n}{8}-2\right)>2 $. If $ D_{1} $ is  noncomplete, then
Lemma \ref{result 0}(ii) implies that  $ w\left(G-S\right)=w(G-N_{G}(uv))=3 $, so $ \left|V(D_{1})\right|\geq n-\left|N_{G}(uv)\right|-2\geq \frac{7}{8}n $. By Lemma \ref{result 8}, the claim holds.

In the following, we assume there exists a clique $ Q_{1}$ in $D_{1}$ such that $ \left|V(Q_{1})\right|-2\left|N_{G}(Q_{1})\right|\geq 2 $. 
Let $ S^{\prime} $ be a largest subset of $N_{G}(Q_{1})$ such that $ \left|N_{Q_{1}}(S^{\prime})\right|<2\left|S^{\prime}\right| $, $ S^{\prime\prime}=N_{G}(Q_{1})-S^{\prime} $ and $ S^{\ast}=N_{Q_{1}}(S^{\prime})\cup S^{\prime\prime} $. Hence for any subset $ T\subseteq S^{\prime\prime} $, $ \left|N_{Q_{1}}(T)\backslash N_{Q_{1}}(S^{\prime})\right|\geq 2\left|T\right| $. By Lemma \ref{result 1}, $ G $ has a $ K_{1, 2} $-matching   with centers as  vertices of $S^{\prime\prime}$ and  partners in $  V(Q_{1})\backslash N_{Q_{1}}(S^{\prime})$. Let  $ D_{1}^{\ast}=Q_{1}-N_{Q_{1}}(S^{\prime}) $. Then $ \lvert V(D_{1}^{\ast})\rvert=\left|V(Q_{1})\right|-\left|N_{Q_{1}}(S^{\prime})\right|> \lvert V(Q_{1})\rvert-2\left|S^{\prime}\right|\geq 
\lvert V(Q_{1})\rvert-2\left|N_{G}(Q_{1})\right|\geq 2 $ and $ D_{1}^{\ast} $ is complete component of $G-S^{\ast}$. Note that $\emptyset\neq S^{\prime}\cup \left(V(G)\backslash\left(V(D_{1})\cup S\right)\right)\subseteq V(G)\backslash\left(V(Q_{1})\cup S^{\ast}\right) $.

\noindent{\bf Claim 1.} 
If the vertices in $G-\left(V(Q_{1})\cup S^{\ast}\right)$ are covered by a set $ F $ of pairwise disjoint paths  of $ G $ such that two ends of each path in $F$ belong to $ S^{\ast} $, then $ G $ has a hamiltonian cycle.

If $ w\left(G-S^{\ast}\right)\geq 3 $, then from the proof of Theorem 1.2 in \cite{Gao}, the theorem is proved. Thus we assume $ w\left(G-S^{\ast}\right)=2 $. Let $D^{\ast}_{2}$ be the other component of $G-S^{\ast}$.
As $ G $ is $(P_{3}\cup 3P_{1})$-free, this implies that $D_{2}^{\ast} $ is $ \left(P_{3}\cup 2P_{1}\right) $-free.

\noindent{\bf Claim 2.} If $ \alpha(D_{2}^{\ast})\leq 4 $, then $ G $ is hamiltonian. 

\begin{proof}
	If $ \alpha(S^{\ast})\leq 4 $, then $ \alpha (G)\leq 9< \kappa(G) $, so by Lemma \ref{result 7}(ii),  $ G $ is hamiltonian. Hence we assume $ \alpha(S^{\ast})\geq 5 $.
	
	\noindent{\bf Case 1.} $ D_{2}^{\ast} $ is hamiltonian.
	
	Let $ C $ be a hamiltonian cycle of $ D_{2}^{\ast} $ and $ I $ be an independent set of $ S^{\ast} $. By $\kappa(G)\geq 30$, there exist 3 disjoint edges $s_{1}w_{1}, s_{2}w_{2}, s_{3}w_{3}$  such that $ s_{i}\in S^{\ast} $ and $w_{i}\in D_{2}^{\ast}$ for eah $i\in [1, 3]$. If there exists an edge  between $\left\{w_{1}, w_{2}, w_{3}\right\}$ in $C$, then  we can obtain a set satisfying Claim 1, which contains only one path. We assume there exists no edge in $C$ between $\left\{w_{1}, w_{2}, w_{3}\right\}$. Let $s_{4}, s_{5}\in I\backslash\left\{s_{1}, s_{2}, s_{3}\right\}$. As each vertex in $ S^{\ast} $ has a neighbor in $D_{1}^{\ast}$, there exists a path of order $3$ in $G[V(D_{1}^{\ast})\cup\left\{s_{4}, s_{5}\right\}]$. By the $(P_{3}\cup 3P_{1})$-freeness, there exists an edge in $G$ between $\left\{w_{1}^{+}, w_{2}^{+}, w_{3}^{+}\right\}$ or $ \left(N_{G}(s_{4})\cup N_{G}(s_{5})\right)\cap \left\{w_{1}^{+}, w_{2}^{+}, w_{3}^{+}\right\}\neq \emptyset $. It is easy to obtain  a set of one path satisfying Claim 1.

	\noindent{\bf Case 2.} $ D_{2}^{\ast} $ is not hamiltonian.
	
	Lemma \ref{result 7}(ii) implies that $ \kappa(D_{2}^{\ast})\leq 3 $.  Let $ k=\kappa(D_{2}^{\ast}) $, $ W $ be a minimum cutset of $ D_{2}^{\ast} $, and $ l=w\left(D^{\ast}_{2}-W\right)$. Then $\left|W\right|\leq 3$ and $l\leq \alpha(D^{\ast}_{2})\leq 4 $.  
	
	We will prove that there exists a set $ A$ of at most 4 pairwise disjoint  paths covering all the vertices in $D^{\ast}_{2}-W$, with two ends of each path in $S^{\ast}$ and all the internal vertices in $D^{\ast}_{2}-W$.
	Since $\left|W\right|\leq 3$ and $\tau(G)\geq 12$, we have that $H=G-W$ is $ 12 $-tough.
	By $\kappa(H)\geq 24$ and Lemma \ref{result 6}, $H$ has a generalized $K_{1, 2}$-matching with centers as components of $H-S^{\ast}$. If each component of $D^{\ast}_{2}-W$ is complete, then clearly, $A$ exists. Thus we assume there exists a noncomplete component $D_{21}$ of $D^{\ast}_{2}-W$. Then by Lemma \ref{result 0}(i), $D_{21}$ is $(P_{3}\cup P_{1})$-free. Lemma \ref{result 0}(ii) implies that $w\left(D^{\ast}_{2}-W\right)=2$. Let $D_{22}$ be the other component of $D^{\ast}_{2}-W$. Then by Lemma \ref{result 0}(i), $D_{22}$ is complete. If $D_{21}$ is hamiltonian, then by a similar argument as Case 1, we can obtain a path covering all the vertices in $D_{21}$ with two ends $s_{1}, s_{2}\in S^{\ast}$ and all the internal vertices in $D_{21}$.  By $\kappa(H)\geq 24$, there exist two disjoint edges between $D_{22}$ and $S^{\ast}\backslash\left\{s_{1}, s_{2}\right\}$. So $A$ exists.	
	If $D_{21}$ is not hamiltonian, then  Theorem \ref{result 10} implies that $\tau(D)<1$. Let 
	$T$ be a tough set of $D_{21}$. By Lemma \ref{result 0}(ii), every component of $D_{21}-T$ is complete. As $\alpha(D_{21})\leq 3$ and $\tau(D_{21})<1$, we have $1\leq \left|T\right|<w\left(D_{21}-T\right)\leq \alpha(D_{21})\leq 3$. Let $H_{1}=G-\left(W\cup T\right).$  Then $ \tau(H_{1})\geq 9 $. By Corollary \ref{C2},
	 $H_{1}$ has a generalized $K_{1, 2}$-matching with centers as components of $D_{2}^{\ast}-(W\cup T)$ and partners in $S^{\ast}$. By Lemma 2.1(iv), we can obtain $w\left(D_{21}-T\right)-\left|T\right|+1$ pairwise disjoint paths covering all the vertices in $D^{\ast}_{2}-W$ with all ends in $S^{\ast}$ and all the internal vertices in $D_{2}^{\ast}-W$.

	Hence $\left|N_{S^{\ast}}(A)\right|\leq 8$. 	Let $z$ be an arbitrary vertex in $W$. If $z$ has at least $9$ neighbors in $D_{2}^{\ast}-W$, then as $\left|A\right|\leq 4$, we have that $z$ is adjacent to two consecutive vertices in some path of $A$, or by $\alpha (D_{2}^{\ast})\leq 4$, there exist two vertices $x_{1}, x_{2}\in N_{G}(z)$ in some path of $A$ such that $x_{1}^{+}x_{2}^{+}\in E(G)$. Now we can insert the vertex $z$ to some path of $A$.
	Otherwise, as $\delta(G)\geq 30$, $k\leq 3$ and $\left|N_{S^{\ast}}(A)\right|\leq 8$, this implies that for each $z\in W$, $z$ has  at least 9 neighbors in $S^{\ast}\backslash N_{S^{\ast}}(A)
	$. Let $\left\{x, y\right\}\subseteq N_{S^{\ast}}(z)\backslash N_{S^{\ast}}(A)$ and $P=xzy$. Then $A\cup\left\{P\right\} $ is a set covering all the  vertices in $G-(W\backslash\left\{w\right\})$. By the repetition of this argument on the vertices of $W$, we can obtain a set satisfying Claim 1, so $G$ is hamiltonian.
\end{proof}


Let $l=\alpha(D_{2}^{\ast})$. By Claim 2, we assume $ l\geq 5 $. Let $ I $ be a maximum independent  set of $ D_{2}^{\ast} $ and  $ U=\left\{x\in N_{D_{2}^{\ast}}(I): \left|N_{I}(x)\right|\geq 2\right\} $. By Lemma \ref{result 0} (iv), we have $ w(D_{2}^{\ast}-U) =w\left(D_{2}^{\ast}-N_{D_{2}^{\ast}(I)}\right)=l\geq 5$.  Let $ B=\left\{B_{1}, \cdots, B_{l}\right\} $ be the set of  components in $D^{\ast}_{2}-U$. By Lemma \ref{result 0}(ii), each component in $B$ is complete. For each $ i\in [1, l] $, let $ Q_{i} $ be a hamiltonian path of $ B_{i} $ with two ends $a_{i}$ and $b_{i}$. By $\kappa(G)\geq 30$, there exist two disjoint edges between $U$ and $S^{\ast}$, or there exist two disjoint edges between $B$ and $S^{\ast}$.   Let \( s_{1}, s_{2} \in S^{\ast} \) be the ends of two disjoint edges, either between \( U \) and \( S^{\ast} \) or between \( B \) and \( S^{\ast} \).
Let $ H=G-\left\{s_{1}, s_{2}\right\} $. Then $H$ is $12$-tough. 
Let $X$ be a maximum subset of $B$ such that $ H $ has a generalized $K_{1, 2}$-matching $M$ with centers as components in $X$ and all the partners in $S^{\ast}$. Let $Y=B\backslash X$ and $y=\left|Y\right|$.


\noindent{\bf Claim 3.} $\left|U\right|\geq 6y$.

\begin{proof}
By $l\geq 5$ and Lemma \ref{result 6}, there exists a $ K_{1, 12} $-matching $ M $ in $H$ with centers as the components of $B$ and partners in $ U\cup \left(S^{\ast}\backslash\left\{s_{1}, s_{2}\right\}\right)$. Let $G_{1}$ be a graph with $V(G_{1})=V(B)\cup (V(M)\cap \left(S^{\ast}\backslash\left\{s_{1}, s_{2}\right\}\right))$ and $E(G_{1})=E(M)\cup E(B)$. Let $X_{1}$ be the maximum subset of $B$ such that $G_{1}$ has a generalized $K_{1, 2}$-matching with centers as $X_{1}$, and $ Y_{1}=B_{1}\backslash X_{1} $. Then $\left|U\right|\geq 6\left|Y_{1}\right|$. 
By the maximality of $X$, we have $ \left|X_{1}\right|\leq \left|X\right| $ and then $\left|Y_{1}\right|\geq y$, so $ \left|U\right|\geq  6y$.
\end{proof}

We will construct a path $ P $ with two ends in $S^{\ast}$ and all the internal vertices in $D^{\ast}_{2}$ that covers all the vertices in at least $ y-2 $ components of $Y$ and all the verices in $U$. Moreover, since $G$ has a generalized $K_{1, 2}$-matching with centers as  components in $X$ and partners in $S^{\ast}$, we can obtain a cycle $C_{1}$ containing all the vertices in $V(D_{1}^{\ast})\cup V(X)\cup S^{\ast}$ by Claim 1.  Then insert the vertices in the remaining components of $Y$ to $C_{1}$. 
We consider two cases below regarding the number of disjoint edges between $D_{2}^{\ast}$ and $S^{\ast}$.

\noindent{\bf Case 1.} There exist 4 disjoint edges between $U$ and $S^{\ast}$.

Assume $ s_{1}, s_{2}\in S^{\ast} $ be two ends of two disjoint edges between $U$ and $ S^{\ast}$.

\noindent{\bf Claim 4.} In \( G[U] \), there exists a set of at most \( l \) pairwise disjoint paths covering all vertices in \( U \), with one path ending at a vertex adjacent to a vertex in \( S^{\ast} \) and another ending at a vertex adjacent to a distinct vertex in \( S^{\ast} \).

\begin{proof}
Let $ A $ be a minimum set of pairwise disjoint paths in $ G[U] $ covering all the vertices in $U$, and $k=\left|A\right|$. 
By Lemma \ref{result 9}, we have $k\leq l$. Let $y_{1}$ and $y_{2}$ be two distint vertices in $U$ such that $s_{1}y_{1}\in E(G)$ and $s_{2}y_{2}\in E(G)$.	
If $k\leq l-3$, then let $P_{i}$ be a path in $A$ containing $y_{i}$ and $e_{i}$ be an edge incident with $y_{i}$ in $P_{i}$ for each $i\in [1, 2]$. Now if $y_{i}$ is an internal vertex of $P_{i}$, then delete $e_{i}$ from $P_{i}$ and denote the  resulting path obtained from $ P_{i} $ by $P_{i}^{\prime}$, otherwise, let $P_{i}^{\prime}=P_{i}$. Then $\left\{P_{1}^{\prime}, P_{2}^{\prime}\right\}\cup \left(A\backslash 
\left\{P_{1}, P_{2}\right\}\right)$ is a desired set.  		
Thus assume $ 3\leq l-2\leq k\leq l $. By Lemma \ref{result 9}, there exists a set $U_{1}\subseteq U$ such that 
$ l-2\leq w\left(G[U]-U_{1}\right)-\left|U_{1}\right|\leq l$. This implies that $\left|U_{1}\right|\leq 2$. By the condition of Case 1, assume $\left\{y_{1}, y_{2}\right\}\subseteq U\backslash U_{1} $.
By Lemma \ref{result 0}(ii), each component of $G[U]-U_{1}$ is complete.  If $s_{1}$  and $s_{2}$ have neighbors in two distinct components of $G[U]-U_{1}$, then as each vertex in $U_{1}$ (if exists) is adjacent to all the vertices at least $w\left(G[U]-U_{1}\right)-1$ components of $ G[U]-U_{1} $ by Lemma \ref{result 0}(iv), we can obtain a desired set. Otherwise, 
$s_{1} $ and $s_{2}$ have neighbors in one common component of $G[U]-U_{1}$.  Lemma \ref{result 0}(iv) implies that $w\left(G[U]-U_{1}\right)=3$, so $U_{1}=\emptyset$.  Let $P_{1}$, $P_{2}$ and $P_{3}$ be  hamiltonian paths of 3 components  of $G[U]-U_{1}$, respectively. Without loss of generality, we assume $y_{1}$ and $y_{2}$ are two ends of $P_{1}$. Delete the edge in $P_{1}$ incident with $y_{1}$, and denote the set of two resulting paths by $A^{\prime}$. Then $A^{\prime}\cup \left(A\backslash \left\{P_{1}\right\}\right)$ is a desired set.
\end{proof}



Let $A_{1}$ be a minimum set satisfying Claim 4. 
Let $A_{1}=\left\{P_{1}, \dots , P_m\right\} $
$$ P_{1}=x_{1}\cdots x_{k_{1}},  P_{2}=x_{k_{1}+1} \cdots x_{k_{2}}, \dots,  P_{m}=x_{k_{m-1}+1} \dots  x_{k_{m}}.$$ Then $m\leq l$.  Assume $ x_{1}s_{1}, x_{k_{m}} s_{2}\in E(G) $,   $\left|V(P_{1})\right|\leq \left|V(P_{m})\right| $ and $\left|V(P_{2})\right|\leq \cdots \leq \left|V(P_{m-1})\right|$. 
If there exists a trivial path in $A_{1}$, then let $ j\in [2, m-1] $ such that $ \left|V(P_{j})\right|=1 $ and $ \left|V(P_{j+1})\right|\geq 2 $; otherwise, let $ j=0 $ and $ k_{j}=0 $. 

\noindent{\bf Claim 5.} If $l=m$, then there exists a set satisfying Claim 1.

\begin{proof}
	From the proof of Claim 4, we have $ w\left(G[U]\right)=m=l\geq 5 $.
	By Lemma \ref{result 0}(ii), each component of $G[U]$ is complete. Let $C_{1}, \dots, C_{m}$ be the components of $G[U]$. For each $i\in [1, m]$, assume $P_{i}$ is a hamiltonian path of $C_{i}$.  Lemma \ref{result 0}(iv) implies that the vertices in at least $ l-1 $ components of $B$ are adjacent to all the vertices in at least $ l-1 $ components of $G[U]$ and each vertex in $U$ is adjacent  to all the vertices in at leasr $l-1$ components of $B$, call this {\bf the adjacency condition 1.} And by Lemma \ref{result 0}(iii), each vertex in a component from $B$ that has a neighbor in $G[U]$ is adjacent to all the vertices in some component of $G[U]$, call this the {\bf adjacency condition 2.}

	By adjacency condition 1, for each $i\in [1, m-1]$, we may
	assume $ \left\{x_{k_{i}}, x_{k_{i}+1}\right\}\sim B_{i} $, and then connect two paths $P_{i}$ and $P_{i+1}$ by adding a path $x_{k_{i}}a_{i}Q_{i}b_{i}x_{k_{i}+1}$. Denote the resulting path by $P$.
	If $G\left[V(B_{l})\cup\left(S^{\ast}\backslash\left\{s_{1}, s_{2}\right\}\right)\right]$ has a generalized $K_{1, 2}$ centered at $B_{l}$, then let $s_{3}$ and $s_{4}$ be two partners of $B_{l}$. Without loss of generality, we ssume $s_{3}a_{l}, s_{4}b_{l}\in E(G)$. Then $\left\{P\right\}\cup \left\{s_{3}a_{l}Q_{l}b_{l}s_{4}\right\}$ is a desired set. Now we assume $G\left[V(B_{l})\cup\left(S^{\ast}\backslash\left\{s_{1}, s_{2}\right\}\right)\right]$ has no generalized $K_{1, 2}$ centered at $B_{l}$. By $\kappa(G)\geq 30$, $G\left[V(B_{l})\cup U\right]$ has a generalized $K_{1, 4}$ centered at $B_{l}$. Let $ \left\{B_{l}^{1}, B_{l}^{2}\right\}$ and $ \left\{T_{1}, T_{2}\right\}$ be two partitions of $B_{l}$ and four partners of $B_{l}$ such that $T_{i}\subseteq N_{G}(B_{l}^{i})$ for $i=1, 2$.   
Let	$T_{1}=\left\{y_{1}, y_{2}\right\}$ and $T_{2}= \left\{z_{1}, z_{2}\right\}$. If $4$ partners of $B_{l}$ are in a common component $C_{r}$ of $G[U]$, then since $C_{r}$ is complete, there exists a hamiltonian path $P_{r}^{\prime}$ of $C_{r}$ such that $x_{k_{r-1}+1}$ and $x_{k_{r}}$ are two ends of $P_{r}^{\prime}$ and $y_{i}z_{j}\in E(P_{r}^{\prime})$, where $y_{i}\in \left\{y_{1}, y_{2}\right\}$ and $z_{j}\in \left\{z_{1}, z_{2}\right\}$. Assume $y_{i}\sim a_{l}$ and $z_{j}\sim b_{l}$. Replace the edge $y_{i}z_{j}$ with the path $y_{i}a_{l}Q_{l}b_{l}z_{j}$. Denote by $P_{r}^{\prime\prime}$ the resulting path obtained from $P_{r}^{\prime}$.  Now replace the subpath $P_{r}$ of $P$ with $P_{r}^{\prime\prime}$. Denote by $P^{\ast}$ the resulting path obtained from $P$. Note that $\left\{P^{\ast}\right\}$ is a desired set. Otherwise, for each $i\in [1, l]$, $G\left[V(B_{i})\cup U\right]$ has a generalized $K_{1, 2}$ centered at $B_{l}$ with 2 partners in 2 distinct components of $G\left[W_{2}\right]$. If $E\left(V(B), S^{\ast}\backslash\left\{s_{1}, s_{2}\right\}\right)\neq\emptyset$, then let $s_{3}\in S^{\ast}\backslash\left\{s_{1}, s_{2}\right\}$ be a vertex adjacent to some vertex in $B$. By adjacency conditions 1 and 2, assume $\left\{s_{3}, x_{1}\right\}\sim B_{1}$, $\left\{x_{k_{i}}, x_{k_{i}+1}\right\}\sim B_{i+1}$ for each $i\in [1, l-1]$. Now join $s_{3}$ and $x_{1}$ to two ends $a_{1}$ and $b_{1}$ of $Q_{1}$, respectively. After that, for each $i\in [1, l-1]$, connect two paths $P_{i}$ and $P_{i+1}$ by adding a path $x_{k_{i}}a_{i+1}Q_{i+1}b_{i+1}x_{k_{i}+1}$. Denote the resulting path by $P^{\prime}$. Note that $\left\{s_{3}a_{1}Q_{1}b_{1}x_{1}P^{\prime}x_{k_{l}s_{2}}\right\}$ is a desired set. If $E\left(V(B), S^{\ast}\backslash\left\{s_{1}, s_{2}\right\}\right)=\emptyset$, then by $\kappa(G)\geq 15$, we have $\frac{\left|U\cup\left\{s_{1}, s_{2}\right\}\right|}{l+1}\geq 15$, so $\left|U\right|\geq 15l+13$. This implies that there exists a component $C_{j}$ of $G[U]$ containing at least 15 vertices. By adjacency conditions 1 and 2, we can assume $\left\{x_{k_{j}-1}, x_{k_{j}}\right\}\sim B_{1}$, 
	and $\left\{x_{k_{i}}, x_{k_{i}+1}\right\}\sim B_{i+1}$ for each $i\in [1, l-1]$. Now add a path $x_{k_{i}}a_{i+1}Q_{i+1}b_{i+1}x_{k_{i}+1}$ to connect two paths $P_{i}$ and $P_{k_{i+1}}$, and then replace the edge $x_{k_{j}-1}x_{k_{j}}$ with the path  $x_{k_{j}-1}a_{1}Q_{1}b_{1}x_{k_{j}}$. Denote the resulting path by $P^{\prime}$. Clearly, $\left\{s_{1}x_{1}P^{\prime}x_{k_{l}}s_{2}\right\}$ is a desired set.
\end{proof}

Now we assume $l\geq m+1$.  Let $ k=\max\left\{y, m+1\right\} $ and $ B^{\ast}=Y\cup X_{1} $, where $X_{1}\subseteq X$ with $\left|B^{\ast}\right|=\left|X_{1}\right|+\left|Y\right|=k$. 
Assume $ B^{\ast}=\left\{B_{1},\dots, B_{k}\right\} $, and $\left\{B_{1},\dots, B_{y-2}\right\}\subseteq Y$ if $y\geq 3$. We will construct a path covering all the vertices in $k-2$ components of $B^{\ast}$ and the vertices in $U$ through the following steps.

\noindent{\bf Step 1.}
If $ m=1 $, then let $ P^{\prime}=s_{1}x_{1}P_{1}x_{k_{1}}s_{2} $.
If $ m\geq 2 $, then by Lemma \ref{result 0}(iv), assume $ \left\{x_{k_{i}}, x_{k_{i}+1}\right\}\sim B_{i} $ for each $i\in [1, m-1]$. Now connect  two path $P_{i}$ and $P_{i+1}$ by adding a path $x_{k_{i}}a_{i}Q_{i}b_{i}x_{k_{i}+1}$. Denote the resulting  path by $P^{\prime}$. 

\noindent{\bf Step 2.} If $y\leq m+1$, then let $P=P^{\prime}$. If $y\geq m+2$, then let $a\in [1, k_{m}]$ such that $ \left|E(x_{1}P^{\prime}x_{a})\cap E(A_{1})\right|=y-m-1 $. Let $ E(x_{1}P^{\prime}x_{a})\cap E(A_{1})=\left\{e_{1}, \dots, e_{y-m-1}\right\} $. For each $i\in [1, y-m-1]$, let $y_{i}$ and $z_{i}$ be two ends of $e_{i}$, and assume $\left\{y_{i}, z_{i}\right\}\sim B_{m-1+i} $ by Lemma \ref{result 0}(iv). Replace each edge $e_{i}$ in $P^{\prime}$ with the path $y_{i}a_{m-1+i}Q_{m-1+i}b_{m-1+i}z_{i}$. Denote by $P$ the resulting path obtained from $P^{\prime}$.

By the construction of $B^{\ast}$, we have $B-B^{\ast}\subseteq X$. Since $H$ has a generalized $K_{1, 2}$-matching with centers as components in $X$ and partners in $S^{\ast}$. For each $i\in [k+1, l]$, let $u_{i}, v_{i}\in S^{\ast}$ be two partners of $B_{i}$. Since $B_{i}$ is complete, assume $u_{i}a_{i}, v_{i}b_{i}\in E(G)$. Let $P^{\prime}_{i}=u_{i}a_{i}Q_{i}b_{i}v_{i}$. In the following, we will prove that there exists a set satisfying Claim 1.

Note that there are exactly two components $ B_{k-1} $ and $B_{k}$ from $B^{\ast}$ not  covered by $P$. If $\left\{B_{k-1}, B_{k}\right\}\subseteq X$, then let $u_{i}, v_{i}\in S^{\ast}$ be two partners of $B_{i}$ for each $i\in \left\{k-1, k\right\}$. Note that $\left\{P\right\}\cup \left(\bigcup_{i=k-1}^{l} \left\{u_{i}a_{i}Q_{i}b_{i}v_{i}\right\}\right)$ is a desired set. Thus we assume $B_{k}\in Y$. Then $y\geq 1$. By the construction of $P$, let $p$ be the minimum number in $[1, k_{m}]$ such that each vertex of $ V(x_{p}Px_{k_{m}})\cap V(A_{1})$ is adjacent in $P$ to some vertex in at most one component from $Y$. Let $ z=\left|V(x_{p}Px_{k_{m}})\cap V(A_{1})\right| $.  We claim that $z\geq \lceil \frac{5y}{2}\rceil+1\geq 4$. By the construction of $P$, we consider the following cases.
If  $ y\leq j+1 $, then  $ z\geq \frac{k_{m}-\left((y-1)-1\right)}{2} $. If $ j+2\leq y\leq m+1 $, then  $z\geq \frac{k_{m}-j+1}{2} $. Now assume $ y\geq m+2 $. Recall that $\left|V(P_{1})\right|=k_{1}$. If $ 1\leq k_{1}\leq y-m-1 $, then $z=\left|E(A_{1})\right|-\left|E(x_{1}P^{\prime}x_{p})\cap E(A_{1})\right|+1\geq k_{m}-m-(y-m-1)+1.$ If $ k_{1}\geq y-m $, then as $ k_{1}\leq \left|V(P_{m})\right| $, we have $z\geq \frac{\left|E(A_{1})\right|}{2}+1\geq \frac{k_{m}-m}{2}+1.$ In each case, since $k_{m}\geq 6y$ by Claim 3, the claim holds.

Assume the drection of $P$ is from $x_{1}$ to $x_{k_{m}}$. Let $x_{i}\in V(A_{1})$. By the construction of $P$, we have that either $\left|V(x_{i} Px_{i+1})\right|=2$ or $P(x_{i}, x_{i+1})$ is a hamiltonian path of some component in $B^{\ast}$. If $P(x_{i}, x_{i+1})$ is a hamiltonian path of a component in $Y$, we call $ x_{i}Px_{i+1} $ a saturated segment. Otherwise, we call $ x_{i}Px_{i+1} $ a unsaturated segment.  Clearly, the following claim holds.
 


\noindent{\bf Claim 6.} Any two saturated segments in $x_{p-1}Px_{k_{m}}  $ contain no common vertices.

\noindent{\bf Case 1.1.} 
For each component $B^{\prime}\in\left\{B_{k-1}, B_{k}\right\}$, $G
\left[V(B^{\prime})\cup \left(\bigcup_{i=p}^{k_{m}}\left\{x_{i}\right\}\right)\right]$ has a generalized $K_{1, 2}$ centered at $B^{\prime}$.

We will prove that there exists a path $ P^{\ast} $ obtained from $P$ in $D^{\ast}_{2}$  that covers all the vertices in $V(Y)\cup U$ and has two ends $x_{1}$ and  $x_{k_{m}}$.
 Then $\left\{s_{1}x_{1}P^{\ast}x_{k_{m}}s_{2}\right\}\cup \left(\bigcup_{i=k+1}^{l} P_{i}^{\prime}\right)$ is a desired set.
 For any two vertices $x_{i}, x_{j}$ in $ A_{1} $, if $j=i+1$, we say that $x_{i}$ and $x_{j}$ are consecutive.

\noindent{\bf Case 1.1.1.}  For each component $B^{\prime}\in\left\{B_{k-1}, B_{k}\right\}$, $G
\left[V(B^{\prime})\cup \left(\bigcup_{i=p}^{k_{m}}\left\{x_{i}\right\}\right)\right]$ has a generalized $K_{1, 2}$ centered at $B^{\prime}$ with two consecutive partners.


Let $x_{i}$ and $x_{i+1}$ be two partners of $B_{k-1}$, and let $x_{j}$ and $x_{j+1}$ be two partners of $B_{k}$. Assume $x_{i}a_{k-1}$, $x_{i+1}b_{k-1}$, $x_{j}a_{k}$, $x_{j+1} b_{k}\in E(G)$. By $ z\geq 4 $ and Lemma \ref{result 0}(iv), we assume $i<j$, and $x_{i-1}\nsim B_{k-1}$ if $i-1\in [p, x_{k_{m}}]$. We consider the following possibilities. 

(i) Both $x_{i}Px_{i+1}$ and $x_{j}Px_{j+1}$ are unsaturated segements.  Replace two subpaths $x_{i}Px_{i+1}$ and $x_{j}Px_{j+1}$ of $P$ with $x_{i}a_{k-1}Q_{k-1}b_{k-1}x_{i+1}$ and $x_{j}a_{k}Q_{k}b_{k}x_{j+1}$, respectively. Denote the resulting path by $P^{\ast}$.

(ii) $x_{i}Px_{i+1}$ is saturated. Assume $V(P(x_{i}, x_{i+1}))\subseteq V(B_{i})$.  By Claim 5, $x_{i-1}Px_{i}$ is unsaturated, this implies that $i-1\geq p$. By our assumption that $x_{i-1}\nsim B_{k-1}$ and Lemma \ref{result 0}(iv), we have $x_{i-1}\sim B\backslash\left\{B_{k-1}\right\}$.  Replace two subpaths $x_{i-1}Px_{i}$ and $x_{i}Px_{i+1}$ of $P$ with $ x_{i-1}b_{i}Q_{i}a_{i}x_{i} $
and $x_{i}a_{k-1}Q_{k-1}b_{k-1}x_{i+1}$, respectively. When $x_{j}Px_{j+1}$ is unsaturated,  replace the path $x_{j}Px_{j+1}$ with $x_{j}a_{k}Q_{k}b_{k}x_{j+1}$.
When $x_{j}Px_{j+1}$ is saturated, assume $P\left(x_{j}, x_{j+1}\right)\subseteq  V(B_{j})$. By Claim 6, we have $j\geq i+2$ and $x_{j-1}Px_{j}$ is unsaturated. By Lemma \ref{result 0}(iv), we have $x_{j-1}\sim B_{k}$ or $B_{j}$. Without loss of generality, assume $x_{j}\sim B_{j}$. Replace two paths $x_{j-1}Px_{j}$ and $x_{j}Px_{j+1}$ with $ x_{j-1}b_{j}Q_{j}a_{j}x_{j} $
and $x_{j}a_{k}Q_{k}b_{k}x_{j+1}$, respectively. Denote the resulting path by $P^{\ast}$.

(iii) $x_{i}Px_{i+1}$ is unsaturated and $x_{j}Px_{j+1}$ is saturated. Assume $P(x_{j-1}, x_{j})\subseteq V(B_{j-1})$. Replace the subpath $x_{i}Px_{i+1}$ of $P$ with $x_{i}a_{k-1}Q_{k-1}b_{k-1}x_{i+1}$.  By Claim 5, we have $x_{j-1}Px_{j}$ and $x_{j+1}Px_{j+2}$ (if exists) are unsaturated. Assume $V\left(P(x_{j}, x_{j+1})\right)\subseteq V(B_{j})$. By Lemma \ref{result 0}(iv), we have $x_{j-1}\sim B_{j} \text{~or~} B_{k}$. Without loss of generality, assume $x_{j-1}\sim B_{j}$.
When $j\geq i+2$ or $j+1< k_{m}$, replace two paths $x_{j-1}Px_{j}$ or $x_{j+1}Px_{j+2}$, and $x_{j}Px_{j+1}$ with $x_{j-1}a_{j}Q_{j}b_{j}x_{j}$ or $x_{j+1}a_{j}Q_{j}b_{j}x_{j+2}$, and $x_{j}a_{k}Q_{k}b_{k}x_{j+1}$, respectively. When $j=i+1=k_{m}$, as $z\geq 4$, we have $i-1\geq p$. By our assumption that $x_{i-1}\nsim B_{k-1}$ and Lemma \ref{result 0}(iv), we have $ x_{i-1}\sim B\backslash\left\{B_{k-1}\right\} $, and $x_{i}\sim B_{k} \text{~or~} B_{j}$. By Claim 6, we have that $x_{i-2}Px_{i-1}$ or $x_{i-1}Px_{i}$ is unsaturated. Assume $x_{i-1}Px_{i}$ is unsaturated and $x_{i}\sim B_{j}$. Rplace the paths $x_{i-1}Px_{i}$ and $x_{j}Px_{j+1}$ with $x_{i-1}a_{j}Qb_{j}x_{i}$ and $x_{j}a_{k}Qb_{k}x_{j+1}$, repectively. Denote the resulting path by $P^{\ast}$.

\noindent{\bf Case 1.1.2.}  There exists a component $B^{\prime}\in \left\{B_{k-1}, B_{k}\right\}$ such that two partners of each generalized $K_{1, 2}$ centered at $B^{\prime}$ in $G
\left[V(B^{\prime})\cup \left(\bigcup_{i=p}^{k_{m}}\left\{x_{i}\right\}\right)\right]$ are not consecutive. 

Assume $B^{\prime}=B_{k-1}$.
Let $x_{i}$ and $x_{j}$ be two partners of a generalized $K_{1, 2}$ centered at $B_{k-1}$ in $G
\left[V(B^{\prime})\cup \left(\bigcup_{i=p}^{k_{m}}\left\{x_{i}\right\}\right)\right]$ such that each vertex in $x_{i}Px_{j}$ has a nonneighbor in $B_{k-1}$ and so is adjacent to all the vertices in $B\backslash\left\{B_{k-1}\right\}$ by Lemma \ref{result 0}(iv).
Without loss of generality, assume $i<j$. By the condition of Case 1.1.2,  if $i-1\geq p$, we have $x_{i-1}\nsim B_{k-1}$. 

If $x_{i-1}\sim B_{k-1}$, then $i-1<p$, this implies that $x_{i-1}Px_{i}$ is saturated. Assume $V\left(P(x_{i-1}, x_{i})\right)\subseteq V(B_{i-1})$. By Claim 6, $x_{i}Px_{i+1}$ is unsaturated. As $z\geq 4$, we have $j\geq i+3$ or $j<k_{m}$. If $j\geq i+3$, then replace the paths $x_{i-1}Px_{i}$ and $x_{i}Px_{i+1}$ with $x_{i-1}a_{k-1}Q_{k-1}b_{k-1}x_{i}$ and $x_{i}a_{i-1}Q_{i-1}b_{i-1}x_{i+1}$, respectively.  When $x_{j-2}Px_{j-1}$ is saturated, assume $V\left\{P(x_{j-2}, x_{j-1})\right\}\subseteq V(B_{j-2})$. By Lemma \ref{result 0}(iv), $x_{j}\in B_{k}$ or $ B_{j-2}$. Assume $x_{j}\sim B_{j-2}$. Replace the paths $x_{j-2}Px_{j-1}$ and   $x_{j-1}Px_{j}$ with $ x_{j-2}a_{k}Q_{k}b_{k}x_{j-1} $ and $ x_{j-1}a_{j-2}Q_{j-2}b_{j-2}x_{j} $, respectively. If $j<k_{m}$, then by the condition of Case 1.1.2 and Lemma \ref{result 0}(iv), we have $x_{j+1}\sim B\backslash\left\{B_{k-1}\right\}$. Let $Q^{\prime}=x_{i}a_{k-1}Q_{k-1}b_{k-1}x_{j}Px_{i+1}a_{k}Q_{k}b_{k}x_{j+1}$. Replace the path $x_{i}Px_{j+1}$ with $x_{i}Q^{\prime}x_{j+1}$. Moreover, if $x_{j}Px_{j+1}$ is saturated,  assume $V\left(P(x_{j}, x_{j+1})\right)\subseteq V(B_{j})$. By Claim 6, $x_{j-1}Px_{j}$ is unsaturated. Replace the path $x_{j-1}Px_{j}$ with $x_{j-1}a_{j}Q_{j}b_{j}x_{j}$. 
Denote the resulting path by $P^{\ast}$.

If $x_{i-1}\nsim B_{k-1}$, then by Lemma \ref{result 0}(iv), we have $x_{i-1}\sim B\backslash\left\{B_{k-1}\right\}$. Let $ Q=x_{i-1}a_{k}Q_{k}b_{k}$\\$x_{j-1}Px_{i}a_{k-1}Q_{k-1}b_{k-1}x_{j} .$  Replace the path $x_{i-1}Px_{j}$ with $x_{i-1}Qx_{j}$.
Firstly, we assume $i-1\geq p$. If $x_{i-1}Px_{i}$ is saturated, the $x_{i-2}Px_{i-1}$ and $x_{i}Px_{i+1}$ are unsaturated. Assume $P(x_{i-1}, x_{i})$ is a hamiltonian path of $ B_{i-1}$.
Relace the path $x_{i}Px_{i+1}$ with $x_{i}a_{i-1}Q_{i-1}b_{i-1}x_{i+1}$.
 When $x_{j-1}Px_{j}$ is saturated, assume $P(x_{j-1}, x_{j})\subseteq V(B_{j-1})$. By Lemma \ref{result 0}(iv), $x_{i-2}\sim B_{k} \text{~or~} B_{j-1}$. Assume $x_{i-2}\sim B_{j-1}$. Replace the path $x_{i-2}Px_{i-1}$ with $x_{i-2}a_{j-1}Q_{j-1}b_{j-1}x_{i-1}$. 
Now we assume $i-1< p$. Then $x_{i-1}Px_{i}$ is saturated. Replace the path $x_{i}Px_{i+1}$ with $x_{i}a_{i-1}Q_{i-1}b_{i-1}x_{i+1}$.  As $z\geq 4$, we have $j\geq i+3$ or $j<k_{m}$. If $j\geq i+3$ and $x_{j-1}Px_{j}$ is saturated,  assume $P(x_{j-1}, x_{j})\subseteq V(B_{j-1})$.
 By Claim 6, we have $x_{j-2}Px_{j-1}$ is unsaturated. Replace the path $x_{j-2}Px_{j-1}$ with $x_{j-2}a_{j-1}Q_{j-1}b_{j-1}x_{j-1}$. If $j<k_{m}$, then by  the condition of Case 1.1.2 and Lemma \ref{result 0}(iv), we have $x_{j+1}\sim B\backslash\left\{B_{k-1}\right\}$. When $x_{j-1}Px_{j}$ is saturated,  assume $P(x_{j-1}, x_{j})\subseteq V(B_{j-1})$.  By Claim 6, $x_{j}Px_{j+1}$ is unsaturated. Replace the path $x_{j}Px_{j+1}$ with $x_{j}a_{j-1}Q_{j-1}b_{j-1}x_{j+1}$. 
Denote the resulting path by $P^{\ast}$.

\noindent{\bf Case 1.2.} There exists a component $B^{\prime}\in\left\{B_{k-1}, B_{k}\right\}$, such that $G
\left[V(B^{\prime})\cup \left(\bigcup_{i=p}^{k_{m}}\left\{x_{i}\right\}\right)\right]$ has no generalized $K_{1, 2}$ centered at $B^{\prime}$.

We assume $B^{\prime}=B_{k}$.  By Lemma \ref{result 0}(iv), there exist at least $z-1\geq 3$ vertices in $\bigcup_{i=p}^{k}\left\{x_{i}\right\}$ adjacent to all the vertices in $B\backslash\left\{B_{k}\right\}$. This implies that for each $i\in [1, l]\backslash\left\{k\right\}$,  $G\left[V(B_{i})\cup \left(\bigcup_{i=p}^{k}\left\{x_{i}\right\}\right)\right]$ has a generalized $K_{1, 2}$ centered at $B_{k-1}$ with two consecutive partners in $\bigcup_{i=p}^{k}\left\{x_{i}\right\}$. Let $ x_{j} $ and $x_{j+1}$ be two consecutive partners of $B_{k-1}$. By a similar argument as Case 1.1.1, we can obtain a path $P^{\ast}$ with two ends $s_{1}$ and $s_{2}$ such that $V(P^{\ast})=V(P)\cup V(B_{k-1})$. 
If $G
\left[V(B_{k})\cup \left(\bigcup_{i=1}^{p}\left\{x_{i}\right\}\right)\right]$ has a generalized $K_{1, 2}$ centered at $B_{k}$,  then let $x_{q}, x_{r}\in \bigcup_{i=1}^{p}\left\{x_{i}\right\} $ be two partners  of a generalized $K_{1, 2}$ centered at $ B_{k} $ such that $ r-q $ is minimum. Let  $Q=s_{1}x_{1}Px_{q}x_{k}Q_{k}x_{k}Px_{q+1}b_{r}Q_{r}a_{r}x_{k+1}Ps_{2}$. If $x_{q}Px_{q+1}$ is saturated, then assume $ V(x_{q}Px_{q+1})\subseteq B_{q}$.  By a similar argument as Case 1.1, we can construct a path $Q^{\ast}$ obtained from $P$ such that $s_{1}$ and $s_{2}$ are two ends of $Q^{\ast}$ and  $V(Q^{\ast})=V(P)\cup V(B_{k-1})\cup V(B_{k})$. Note that $\left\{Q^{\ast}\right\}\cup \left(\bigcup_{i=k+1}^{l}\left\{P_{i}^{\prime}\right\}\right)$ is a 
set satisfying Claim 1. Otherwise, by $\kappa(G)\geq 30$, $G\left[V(B_{k})\cup S^{\ast}\backslash\left\{s_{1}, s_{2}\right\}\right]$ has a generalized $K_{1, 2}$ centered at $ B_{k} $. Let $s_{3}$ and $s_{4}$ be two partners of $B_{k}$. Assume $s_{3}\sim a_{k}$ and $s_{4}\sim b_{k}$. Let $ P_{k}^{\prime}=s_{3}a_{k}Q_{k}b_{k}s_{4} $ and  $U=\left\{u_{k+1}, v_{k+1}, \dots, u_{l}, v_{l}\right\}$.  If $ \left\{s_{3}, s_{4}\right\}\cap U=\emptyset $, then $\left\{P^{\ast}\right\}\cup \left(\bigcup_{i=k+1}^{l}\left\{P_{i}^{\prime}\right\}\right)\cup P_{k}^{\prime}$ is a desired set. If $\left|\left\{s_{3}, s_{4}\right\}\cap U\right|=1 $, then assume $ s_{4}=u_{k+1}$, hence  $\left\{P^{\ast}\right\}\cup \left(\bigcup_{i=k+1}^{l}\left\{P_{i}^{\prime}\right\}\right)\cup \left\{s_{3}P_{k}^{\prime}u_{k+1}P_{k+1}^{\prime}v_{k+1}\right\}$ is a desired set.
If $s_{3}$ and $s_{4}$ are partners of one component $B_{r}\in B\backslash B^{\ast}$, then replace the role of $B_{k}$ in Case 1.1 with $B_{r}$, we can obtain a desired set. If $s_{3}$ and $s_{4}$ are partners of two components in $ B\backslash B^{\ast}$, then assume $s_{3}=v_{k+1}$ and $s_{4}=u_{k+2}$, hence $\left\{P^{\ast}\right\}\cup \left(\bigcup_{i=k+3}^{l}\left\{P_{i}^{\prime}\right\}\right)\cup \left\{u_{k+1}P_{k+1}^{\prime}v_{k+1}P_{k}^{\prime}u_{k+2}P_{k+1}^{\prime}v_{k+2}\right\}$ is a desired set.

\noindent{\bf Case 2.} There exists no two disjoint edges between $ U$ and $S^{\ast}$.

$\kappa(G)\geq 30$ implies that there exist two disjoint edges between $S^{\ast}$ and $B$.  Assume $ s_{1}$ and $s_{1} $ are two vertices in $S^{\ast}$ adjacent to vertices in $B$. By Lemma \ref{result 9}, there exists a set $A_{1}$ of at most $l$ disjoint paths in $G[U]$ covering all the vertices in $U_{1}$.
  Let $ A_{1}=\left\{P_{1}, \dots, P_{m}\right\} $, and
$$ P_{1}=x_{1} \cdots x_{k_{1}},  P_{2}=x_{k_{1}+1} \cdots x_{k_{2}} , \dots, P_{m}=x_{k_{m-1}+1} \cdots x_{k_{m}} .$$ Assume  $\left|V(P_{1})\right|\leq \cdots \leq \left|V(P_{m})\right|$. If there exists a trivial path,  let $ j\in [1, m-1] $ such that $ \left|V(P_{j})\right|=1 $ and $ \left|V(P_{j+1})\right|\geq 2 $.
We will construct a path $ P $ by applying the following operations.

\begin{enumerate}[(1)]
\item 
Assume $ l\geq m+3 $. 
Let $ k=\max\left\{5, y, m+3\right\} $ and $ B^{\ast}=Y\cup X_{1} $, where $X_{1}\subseteq X$ with $\left|X_{1}\right|+\left|Y\right|=k$. Assume $ B^{\ast}=\left\{B_{1},\dots, B_{k}\right\} $.
By Lemma \ref{result 0}(iv), assume $ \left\{s_{1}, x_{1}\right\}\sim B_{1} $ and $ \left\{s_{2}, x_{k_{m}}\right\}\sim B_{k-2} $.

\begin{enumerate}[(a)]
\item Assume $m=1$. Let $ P^{\prime}=s_{1}a_{1}Q_{1}b_{1}x_{1}P_{1}x_{k_{1}}a_{k-2}Q_{k-2}b_{k-2}s_{2} $.
If $y\leq 2$, then let $P=P^{\prime}$. If $3\leq y\leq 4$, then  by Lemma \ref{result 0}(iv), assume $\left\{x_{1}, x_{2}\right\}\sim B_{2}$. Replace the edge $x_{1}x_{2}$ in $P^{\prime}$ with the path $x_{1}a_{2}Q_{2}b_{2}x_{2}$. If $y\geq 5$, then by Lemma \ref{result 0}(iv), assume $\left\{x_{i}, x_{i+1}\right\}\sim B_{i+1}$ for each $i\in [1, y-m-3]$. Replace each edge $x_{i}x_{i+1}$ in $P^{\prime}$ with the path $x_{i}a_{i+1}Q_{i+1}b_{i+1}x_{i+1}$. In each case, denote by $P$ the resulting path obtained from $P^{\prime}$.

\item If $ m\geq 2 $, then  for each $i\in [1, k-3]$, assume $ \left\{x_{k_{i}}, x_{k_{i}+1}\right\}\sim B_{i+1} $ by Lemma \ref{result 0}(iv). Then connect two paths $P_{i}$ and $P_{i+1}$ by adding a path $x_{k_{i}}a_{i+1}Q_{i+1}b_{i+1}x_{k_{i}+1}$.  Denote the resulting  path by $P^{\prime}$. Let $P^{\prime\prime}=s_{1}a_{1}Q_{1}b_{1}x_{1}P^{\prime}x_{k_{m}}a_{k-2}Q_{k-2}b_{k-2}s_{2}$. If $ y\leq m+3 $, then let $P=P^{\prime\prime}$. If $y\geq m+4$, then let  $ a\in [j+1, k_{m}] $ such that $ \left|E(x_{k_{1}}P^{\prime\prime}x_{a})\cap E(A_{1})\right|=y-m-3 $. and let $ E(x_{k_{1}}P^{\prime\prime}x_{a})\cap E(A_{1})=\left\{e_{1}, \cdots, e_{y-m-3}\right\} $. 
 For each $i\in [1, y-m-3]$, let $y_{i}$ and $z_{i}$ be two ends of $e_{i}$, and assume $\left\{y_{i}, z_{i}\right\}\sim B_{m+i} $ by Lemma \ref{result 0}(iv). Replace each edge $e_{i}$ in $P^{\prime\prime}$ with $y_{i}a_{m+i}Q_{m+i}b_{m+i}z_{i}$. Denote by $P$  the resulting path obtained from $P^{\prime\prime}$.
\end{enumerate}


Note that the vertices in at most three components of $B^{\ast}$
are not covered by $  P$, of which at most two are in $Y$. Assume $B_{k}\in Y$ is not covered by $P$. Then $y\geq 1$. Let $p$ the minimum number in $[1, k_{m}]$ such that each vertex in $V(x_{p}Px_{k_{m}})$ is adjacent to at most one component of $Y$. Let $z=\left|V(x_{p}P^{\prime}x_{k_{m}})\cap V(A_{1})\right|$. We claim that $z\geq 5y+1\geq 6$. By the construction of $P$, we consider the following cases.
If  $ y\leq j+1 $, then  $ z\geq k_{m}-(y-1) $. If $ j+2\leq y\leq m+3 $, then  $z\geq k_{m}-j+1\geq k_{m}-y+3 $. If $ y\geq m+4 $, then $z=\left|E(A_{1})\right|-\left|E(x_{1}P^{\prime}x_{p})\cap E(A_{1})\right|+1\geq k_{m}-m-(y-m-3)+1.$  By Claim 3, we have $k_{m}\geq 6y$. So in each case, the claim holds.
By a similar argument as Case 1, we can obtain a desired set.



\item Assume $3\leq l-2\leq m $. By Lemma \ref{result 9}, there exists a cutset $ U_{1} $ of $G[U]$ such that $ m\leq w\left(G[U]-U_{1}\right)-\left|U_{1}\right|\leq l $. Let $ m_{1}=w\left(G[U]-U_{1}\right) $. By Lemma \ref{result 0}(iv), the vertices in at least $ l-1 $ components of $D_{2}^{\ast}-U$ are adjacent to all the vertices in at least $ m_{1}-1 $ components of $G[U]-U_{1}$. Let $ B^{\ast}=B $. 

\begin{enumerate}[(a)]
		\item If $ m_{1}=l $, then $ \left|U_{1}\right|\leq 2 $. Let $W$ be a set of vertices in $B$ adjacent to vertices in at most one component of $G[U]-U_{1}$. By Lemma \ref{result 0}(iv), we can obtain $w\left(D_{2}^{\ast}-\left(\left(V(B)\backslash W\right)\cup U_{1}\right)\right)\geq l$.	
	Replacing the role of $B$ and $G[U]$ in Case 1 with $G[\left(U\backslash U_{1}\right)\cup W]$ and $G[\left(V(B)\backslash W\right)\cup U_{1}]$, we can obtain a desired set. 
	
	\item If $ m_{1}=l-2 $, then $ U_{1}=\emptyset $.  Assume $ \left\{x_{1}, s_{1}\right\}\sim B_{1} $, $\left\{s_{2}, x_{k_{m}}\right\}\sim B_{l-1}$ and $\left\{x_{k_{i}}, x_{{k_{i}}+1}\right\}\sim B_{i}$ for each $i\in [1, l-2]$.  
	Let 	$ P^{\prime}=s_{1}a_{1}Q_{1}b_{1}x_{1}P_{1}x_{k_{1}} a_{2}Q_{2}b_{2}x_{k_{1}+1}\cdots$\\$ x_{k_{m}}a_{l-1}Q_{l-1}b_{l-1}s_{2}.$ There exist exactly 1 component of $B$ not covered by $P^{\prime}$. By a similar argument as Claim 5, we can obtain a set satisfying Claim 1.
	

	\item If $ m_{1}=l-1 $, then $ \left|U_{1}\right|\leq 1 $.
	If for each component $B_{i}$ of $B$, $G[V(B_{i})\cup U]$ has a generalized $K_{1, 2}$ centered at $B_{i}$ with two partners in two distinct components of   $G[U]-U_{1}$, then assume $ \left\{x_{1}, s_{1}\right\}\sim B_{1} $, $\left\{s_{2}, x_{k_{m}}\right\}\sim B_{l}$ and $\left\{x_{k_{i}}, x_{{k_{i}}+1}\right\}\sim B_{i}$ for each $i\in [1, l-1]$. Let 	$ P^{\prime}=s_{1}a_{1}Q_{1}b_{1}x_{1}P_{1}x_{k_{1}}a_{2}Q_{2}b_{2}x_{k_{1}+1}\cdots x_{k_{m}}a_{l}Q_{l}b_{l}s_{2}.$ Note that $P^{\prime}$ covers all the vertices in $B$. 	If $U_{1}\neq \emptyset$, then by Lemma \ref{result 1}(iv), we can insert the vertices in $U_{1}$ to $P^{\prime}$. Deonte the result path by $P$. Note that $\left\{P\right\}$ is a desired set.
	
	
Now we assume $G[V(B_{l})\cup U]$ has no generalized $K_{1, 2}$ centered at $B_{l}$ with two partners in two distinct components of $G[U]-U_{1}$. Then there exists a vertex $x\in V(B_{l})$ such that $x$ has neighbors in at most one component of $G[U]-U_{1}$. Assume $N_{U}(x)\subseteq V(A_{l-1})\cup U_{1}$. By Lemma \ref{result 0}(iv), the vertices in $B\backslash\left\{B_{l}\right\}$ are adjacent to all the vertices in $ U\backslash \left(V(A_{l-1})\cup U_{1}\right) $.  Assume $ \left\{x_{1}, s_{1}\right\}\sim B_{1} $, $\left\{s_{2}, x_{k_{m}}\right\}\sim D_{l-1}$ and $\left\{x_{k_{i}}, x_{{k_{i}}+1}\right\}\sim B_{i}$ for each $i\in [1, l-2]$. Let 	$ P^{\prime}=s_{1}a_{1}Q_{1}b_{1}x_{1}P_{1}x_{k_{1}}a_{2}Q_{2}b_{2}x_{k_{1}+1}\cdots$\\$ x_{k_{m_{1}}}a_{l-1}Q_{l-1}b_{l-1}s_{2}.$ Note that $P^{\prime}$ covers all the vertices in $B\backslash\left\{B_{l}\right\}$ and $U\backslash U_{1}$.
By a similarly argument as Case 1, we can obtain a desired set. \qed
\end{enumerate}	
\end{enumerate}

\end{document}